\let\ORIlabel\label
\let\ORIrefstepcounter\refstepcounter
\AddToHook{package/hyperref/before}{%
  \let\label\ORIlabel
  \let\refstepcounter\ORIrefstepcounter}
\documentclass[final,oneeqnum,onetabnum,onefignum,onethmnum]{siamart220329}

\usepackage{amssymb,mathtools}
\usepackage{graphicx}
\usepackage{booktabs}
\usepackage{multirow}
\usepackage{placeins}
\usepackage{microtype}
\usepackage{cite}

\newsiamremark{remark}{Remark}
\newsiamthm{assumption}{Assumption}

\newcommand{\R}{\mathbb R}
\newcommand{\Sph}{\mathbb S}
\newcommand{\norm}[1]{\left\lVert#1\right\rVert}
\newcommand{\abs}[1]{\left\lvert#1\right\rvert}
\newcommand{\erf}{\operatorname{erf}}

\newcommand{\supp}{\operatorname{supp}}

\title{Optimal Sobolev Approximation by Deterministic and Random
Shallow Sigmoidal Networks}
\headers{Sobolev Approximation by Sigmoidal Networks}
{Z. Fu and Y. Wang}
\author{
Zhaohui Fu\thanks{Institute for Mathematical Sciences,
National University of Singapore, Singapore 119076.
Email: \href{mailto:fuzhmath@gmail.com}{fuzhmath@gmail.com}.}
\and
Yangshuai Wang\thanks{Department of Mathematics,
National University of Singapore, Singapore 119076.
Email: \href{mailto:yswang@nus.edu.sg}{yswang@nus.edu.sg}.}
}

\begin{document}
\maketitle

\begin{abstract}
Shallow networks with prescribed or randomly sampled hidden parameters are
widely used as numerical trial spaces, yet their optimal Sobolev approximation
power with standard smooth sigmoidal activations in general dimension remains
unresolved. We establish the corresponding optimal rates for a class of smooth
sigmoidal activations with Schwartz-class derivative decay, including \(\tanh\),
the logistic sigmoid, and the error function \(\erf\).
We first construct deterministic direction--offset dictionaries with \(M\)
features such that every \(u\in H^k(\Omega)\) can be approximated with error of
order \(M^{-(k-m)/d}\) in \(H^m(\Omega)\) for all \(0\le m\le k\). This rate
is optimal in the sense of Kolmogorov widths for Sobolev balls. We further
prove that dictionaries obtained by independent parameter sampling from any
prescribed density bounded away from zero attain the same approximation
exponent with high probability, up to logarithmic oversampling. The analysis
develops a sigmoidal ridge representation and combines it with deterministic or
probabilistic quadrature in direction--offset space while retaining polynomial
control of the output coefficients. Numerical experiments across a broad range
of dimensions, target regularities, and Sobolev error norms recover the
predicted algebraic rates for both deterministic and random feature
dictionaries.

\end{abstract}



\section{Introduction}

The approximation accuracy of neural networks has long been a central
question in their mathematical analysis, as it determines how efficiently
these models can represent functions of prescribed regularity.  In this
work, we study the approximation rates for shallow networks with
sigmoidal activations in the fixed-feature setting, where the hidden layer
is prescribed deterministically or sampled randomly and only the
coefficients in the output layer are fitted.

Classical approximation results already show that shallow sigmoidal networks
can attain the optimal Sobolev exponent when their hidden parameters are
optimized as part of the approximation.  For example, Mhaskar
proved that, for a Sobolev target of smoothness \(s\), a shallow network
with a smooth sigmoidal activation and width \(M=O(N^d)\) can attain the
\(L^p\)-error \(O(N^{-s})=O(M^{-s/d})\)~\cite{MhaskarSigmoid}.
It remains unknown whether the same regularity-dependent rate is attainable
for networks with standard smooth sigmoidal activations when their hidden
parameters are sampled independently.

From a computational perspective, fixed-feature shallow networks have been used
as flexible numerical trial spaces for partial differential equations, with
the hidden parameters prescribed or sampled in advance. The literature employs
several approaches to determine the output coefficients, including collocation
schemes, least-squares procedures, and methods based on weak formulations.
Physics-informed and collocation-based extreme learning machines have been
developed for stationary, evolutionary, nonlinear, and sharp-gradient
problems, while localized and high-dimensional variants address
localization and scalability
\cite{DwivediSrinivasanPIELM,CalabroFabianiSiettosELM,
DongLiLocELM,WangDongHighDimELM}.
Randomized neural trial spaces have also been combined with
Petrov--Galerkin and discontinuous Galerkin formulations
\cite{ShangWangSunRNNPG,SunDongWangLRNNDG}.
In parallel, the random feature method has been developed for stationary
and time-dependent equations, interface problems, and multiscale transport,
with complementary work addressing the conditioning and high-precision
solution of the resulting least-squares systems
\cite{ChenChiEYangRFM,ChenELuoTimeRFM,ChenESunHighPrecisionRFM,
ChenMaWuAPRFM,ChiChenYangInterfaceRFM}.
These developments demonstrate the computational scope of fixed-feature
PDE discretizations, but do not by themselves determine the approximation
properties of the underlying trial spaces.

From an approximation-theoretic perspective, Barron-type integral representations
give dimension-independent \(M^{-1/2}\) approximation rates for
Fourier-moment classes~\cite{BarronSigmoid}. Related kernel random-feature
and quadrature theory expresses sampling complexity through the spectrum of
an associated integral operator~\cite{BachRFQuadrature}; see
\cite{DeVoreHaninPetrova} for a broader account of neural-network
approximation theory.
For ReLU random features, Gonon proved an expected uniform approximation
bound of Monte Carlo order \(M^{-1/2}\) for a Fourier-integrability
class~\cite{GononRFPDE}.
Ming and Yu obtained algebraic and spectral rates for random Fourier features
in elliptic problems, but their analysis is restricted to one
dimension~\cite{MingYuRFM}.
For shallow \(\operatorname{ReLU}^r\) networks, Mao, Siegel, and Xu proved
the classical \(M^{-s/d}\) order up to
\(s=r+(d+1)/2\)~\cite{MaoSiegelXuReLU}, and subsequent fixed-feature upper
and lower bounds identify the saturation rate
\(M^{-[r+(d+1)/2]/d}\)
\cite{LiuMaoXuIntegral,MaoXuSaturation}.
Although smooth sigmoidal activations such as \(\tanh\) have been highly
effective in numerical practice, the existing literature does not provide a
corresponding theory establishing the full regularity-dependent Sobolev
approximation hierarchy for independently sampled features with such
activations in general dimension.

We construct shallow-network spaces for a class of smooth sigmoidal
activations, including \(\tanh\), the
logistic sigmoid, and \(\erf\).  With \(M\) deterministically prescribed
features, every \(u\in H^k(\Omega)\) admits an approximant
\(v_M^{\rm det}\) satisfying, simultaneously for \(m=0,\ldots,k\),
\[
 \|u-v_M^{\rm det}\|_{H^m(\Omega)}
 \lesssim M^{-(k-m)/d}\|u\|_{H^k(\Omega)}.
\]
When the hidden parameters are instead sampled independently from a
prescribed distribution whose density is bounded away from zero, we
prove that, with probability at least \(1-\delta\), every
\(u\in H^k(\Omega)\) admits an approximant satisfying
\[
 \|u-v_M^{\rm ran}\|_{H^m(\Omega)}
 \lesssim
 \left(\frac{M}{\log(M/\delta)}\right)^{-(k-m)/d}
 \|u\|_{H^k(\Omega)},\qquad m=0,\ldots,k.
\]
Hence, i.i.d.\ smooth sigmoidal features preserve the full
regularity-dependent Sobolev approximation rate, up to logarithmic
oversampling.

In addition to the rate estimates, the smooth sigmoidal construction adapts
to any prescribed finite target regularity and yields the corresponding
hierarchy of \(H^m(\Omega)\) estimates, without the loss of conformity that
generally occurs for \(\operatorname{ReLU}^r\) trial spaces when \(m>r\). The
analysis also identifies an activation-dependent hidden-parameter scale
\(\sigma_N\), including an essential logarithmic correction to \(N^{-1}\).
This scale prescribes the magnitude \(\sigma_N^{-1}\) of the hidden weights
and biases sufficient to attain a prescribed Sobolev approximation order,
providing a theoretically motivated reference scale for neural-network 
initialization. The deterministic order matches the
Kolmogorov-width lower bound for Sobolev balls and is therefore optimal among
linear spaces of the same dimension.  The random construction reaches the
same benchmark up to a logarithmic factor.
In both constructions, the output coefficients can be chosen to grow at most
polynomially with the resolution.

To complement the theoretical results, we conduct numerical experiments for
both deterministic and random feature spaces across several regimes. These
experiments cover dimensions \(2\), \(3\), and \(10\), error norms \(L^2\),
\(H^1\), and \(H^2\), and target regularities \(H^k\) for
\(k=2,\ldots,6\), with additional \(H^{10}\) and \(H^{20}\) targets in
\(d=10\). Across these settings, the log--log error curves exhibit algebraic
slopes that closely match the predicted orders.  To the best of our knowledge,
these are the first such rate-resolving experiments for standard smooth
sigmoidal features covering a broad range of target regularities in both
low- and high-dimensional settings.

Section~2 states the main deterministic and random approximation theorems
and summarizes the numerical evidence.  Section~3 presents concluding remarks,
including a discussion of implications and extensions.
Section~4 contains the complete analysis, from the sigmoidal integral
representation and activation estimates through deterministic cubature and
the random matrix argument.  The numerical settings are collected in
Appendix~A.

\section{Main results}
\label{sec:main-results}

This section presents the main approximation results for deterministic and
independently sampled feature dictionaries. After introducing the common
setting, we state the corresponding Sobolev error and coefficient bounds,
establish the optimality of the deterministic rate through a Kolmogorov-width
lower bound, and present numerical evidence for the predicted dependence on
target regularity, error norm, and spatial dimension.

\subsection{Setting}

Fix integers \(d\geq2\), \(k\geq1\), a bounded Lipschitz domain
\(\Omega\Subset\R^d\), and a fixed offset bound
\(B>\sup_{x\in\Omega}|x|\).
Let \(\phi\) be a smooth sigmoidal activation and, for each integer \(N\geq2\), let \(\sigma_N>0\) denote the associated inner-parameter
scale.  Consider
\begin{equation}
 v_N(x)=a_{0,N}+\sum_{\nu=1}^{M_{\omega,N}}
       \sum_{\ell=-L_N}^{L_N}a_{\nu\ell,N}
 \phi\!\left(\frac{\omega_{\nu,N}\cdot x-b_{\ell,N}}{\sigma_N}\right),
 \qquad \omega_{\nu,N}\in\Sph^{d-1},
 \tag{2.1}\label{eq:det-network}
\end{equation}
Here \(M_{\omega,N}\) and \(2L_N+1\) count the directions and offsets,
respectively, \(b_{\ell,N}\in[-B,B]\), and
\(a_{0,N},a_{\nu\ell,N}\in\R\).
The relevant activations are bounded smooth sigmoids with
exponentially decaying derivative profiles, such as \(\tanh\), the logistic
sigmoid, and \(\erf\).  Throughout this section, the activation \(\phi\) and
the scale \(\sigma_N\) are assumed to satisfy the precise hypotheses
stated and verified in Subsection~\ref{sec:activation-setting}.

\begin{remark}
Throughout the paper, \(C,c>0\) denote generic constants that may vary
from line to line and depend only on the fixed problem and construction data
(including the positive lower density bound in random statements), but are
independent of the target function, resolution, feature count, confidence
level, and samples.
\end{remark}

\subsection{Deterministic approximation}

\begin{theorem}[Deterministic sigmoidal-network approximation]
\label{thm:main}
Suppose that the activation and inner-parameter scale hypotheses of
Subsection~\ref{sec:activation-setting} hold.
Then there exist $C,N_0>0$ such that, for every $N\geq N_0$, a network of the form
\eqref{eq:det-network} can be chosen with equal-weight spherical-design
directions, uniformly spaced offsets in a fixed bounded interval, and
\(M_{\omega,N}\leq C N^{d-1}\), \(2L_N+1\leq C N\).  Write
\(M_N:=M_{\omega,N}(2L_N+1)\leq C N^d\) for its number of nonconstant
activations.  Let \(V_N\) be the activation-spanned space associated with
\eqref{eq:det-network}.  These choices
define a linear approximation operator
\(\mathcal T_N^{\mathrm{det}}:H^k(\Omega)\to V_N\).  For every
$u\in H^k(\Omega)$, the network \(v_N=\mathcal T_N^{\mathrm{det}}u\)
satisfies, simultaneously for $m=0,\ldots,k$,
\begin{equation}
 \norm{u-v_N}_{H^m(\Omega)}
 \leq C N^{m-k}\norm u_{H^k(\Omega)}
 \leq C M_N^{-(k-m)/d}\norm u_{H^k(\Omega)},
 \tag{2.2}\label{eq:det-rate}
\end{equation}
For some finite activation-dependent exponent \(\Gamma_{\phi,d,k}\), the
coefficients also satisfy the non-sharp polynomial bound
\begin{equation}
 \norm{(a_{\nu\ell,N})}_{\ell^\infty}
 \leq C N^{\Gamma_{\phi,d,k}}\norm u_{H^k(\Omega)}.
 \tag{2.3}\label{eq:det-coeff-bound}
\end{equation}
\end{theorem}

\subsection{Random features and optimality}

The random theorem replaces the deterministic parameter rule by independent
sampling over the full direction--offset space.

Let ${\cal C}=\Sph^{d-1}\times[-B,B]$, equipped with the product measure
$\mu$ of surface and Lebesgue measure.  For $p=(\omega,b)\in{\cal C}$, set
\(\phi_{N,p}(x):=\phi((\omega\cdot x-b)/\sigma_N)\).  For
\(M\geq1\) and \(X_M=\{p_1,\ldots,p_M\}\subset{\cal C}\), define the
activation-spanned space generated by the random parameter set \(X_M\) as
\(V_{M,N}(X_M):=\operatorname{span}\{1,\phi_{N,p_j}:1\leq j\leq M\}\).

\begin{theorem}[Random sigmoidal-feature approximation]
\label{thm:random}
Suppose that the activation and inner-parameter scale hypotheses of
Subsection~\ref{sec:activation-setting} hold, and let
$p_j=(\omega_j,b_j)$ be independent samples on ${\cal C}$ with density
$\rho$ relative to $\mu$, where \(\rho(p)\geq\rho_0>0\) for
$\mu$-almost every $p$.
There are constants $C,c,N_0$, depending only on the fixed data
and \(\rho_0\), such that, for every integer \(N\geq N_0\) and every
\(0<\delta<1\), if \(M\geq C N^d\log(N/\delta)\), then there is a
sample-dependent linear approximation operator
\(\mathcal T_{N,X_M}:H^k(\Omega)\to V_{M,N}(X_M)\), defined to be zero
whenever the stable-sampling event fails, such that, with
\begin{equation}
 \mathcal E_{N,M}:=\left\{
 \max_{0\leq m\leq k}N^{k-m}
 \sup_{0\ne u\in H^k(\Omega)}
 \frac{\|u-\mathcal T_{N,X_M}u\|_{H^m(\Omega)}}
      {\|u\|_{H^k(\Omega)}}
 \leq C\right\},
 \qquad \mathbb P(\mathcal E_{N,M})\geq1-\delta.
 \tag{2.4}\label{eq:random-N-rate}
\end{equation}
Consequently, provided $M\geq C N_0^d\log(N_0/\delta)$, if
$N=N(M,\delta)\geq N_0$ is chosen as the largest integer satisfying
$M\geq C N^d\log(N/\delta)$, then on \(\mathcal E_{N,M}\),
\(v=\mathcal T_{N,X_M}u\) satisfies
\begin{equation}
 \norm{u-v}_{H^m(\Omega)}
 \leq
 C\left(\frac{M}{\log(M/\delta)}\right)^{-(k-m)/d}
 \norm u_{H^k(\Omega)},\qquad m=0,\ldots,k.
 \tag{2.5}\label{eq:random-width-rate}
\end{equation}
On this event, writing \(v=a_0+\sum_{j=1}^Ma_j\phi_{N,p_j}\), the
coefficients depend linearly on \(u\) and, for some finite
activation-dependent exponent
\(\gamma_{\phi,d,k}\), satisfy
\begin{equation}
 |a_0|+\sum_{j=1}^M|a_j|
 \leq C N^{\gamma_{\phi,d,k}}\|u\|_{H^k(\Omega)}.
 \tag{2.6}\label{eq:random-coeff-bound}
\end{equation}
More precisely, before choosing $M$ in terms of $\delta$, the event can be
chosen so that
\(
 \mathbb P(\mathcal E_{N,M}^{\,c})
 \leq CN^d e^{-cM/N^d}.
\)
\end{theorem}

The following width estimate shows that the approximation order obtained
above is sharp among all finite-dimensional linear
spaces.

\begin{proposition}
\label{prop:linear-width-optimality}
Let $0\leq m<k$ be integers and let
\({\cal B}_k:=\{u\in H^k(\Omega):\|u\|_{H^k(\Omega)}\leq1\}\).
For $M\geq1$, define the Kolmogorov width
\begin{equation}
 d_M({\cal B}_k;H^m(\Omega))
 =
 \inf_{\substack{W\subset H^m(\Omega)\\ \dim W\leq M}}
 \ \sup_{u\in{\cal B}_k}\ \inf_{w\in W}
 \|u-w\|_{H^m(\Omega)}.
 \tag{2.7}
\end{equation}
Classical \(n\)-width theory for Sobolev classes gives the lower bound of
optimal order: for some $c=c(d,k,m,\Omega)>0$ and all
sufficiently large $M$,
\begin{equation}
 d_M({\cal B}_k;H^m(\Omega))
 \geq cM^{-(k-m)/d}.
 \tag{2.8}
\end{equation}
See, for example, Pinkus~\cite[pp.~232--247]{PinkusWidths}.
Consequently, the rate in Theorem~\ref{thm:main} is sharp among all
finite-dimensional linear spaces.  Including the output constant changes
the dimension from $M$ to at most $M+1$ and does not affect the rate.
\end{proposition}

\subsection{Numerical evidence}
\label{sec:numerics}

The experiments provide numerical evidence for the predicted Sobolev
approximation exponent \((k-m)/d\) and illustrate its dependence on the
target-regularity parameter \(k\), the Sobolev order \(m\) of the error norm,
and the spatial dimension \(d\).  The experiments cover deterministic and
random dictionaries and employ both \(\tanh\) and \(\erf\)
activations.  Following standard numerical testing practice, the targets are
generated as finite random Fourier series with Gaussian coefficients and
prescribed Sobolev spectral decay.  For the numerical implementation, the
output coefficients are obtained solely from function values by discrete
\(L^2\) least squares, rather than from the constructive coefficients used in
the analysis.  Full numerical settings are summarized in Appendix~A.

\begin{figure}[htbp]
\centering
\includegraphics[width=\textwidth]{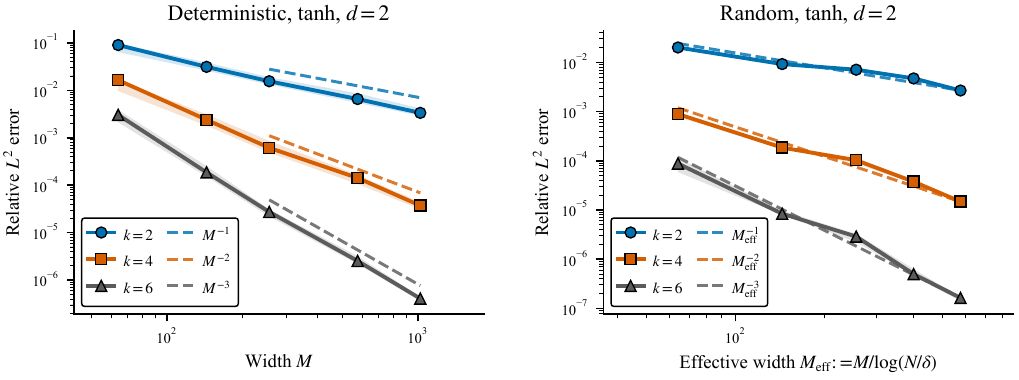}
\caption{Two-dimensional relative \(L^2\) errors for deterministic (left) and
random (right) \(\tanh\) dictionaries.  Curves and shaded regions
show medians and interquartile ranges over target realizations (left) or
dictionary realizations (right).  Dashed lines have order \(k/2\).}
\label{fig:num-2d}
\end{figure}
\FloatBarrier

The numerical results in Figure~\ref{fig:num-2d} clearly exhibit the
\(k/2\) convergence hierarchy predicted by the \(m=0\) cases of
Theorems~\ref{thm:main} and~\ref{thm:random} for both deterministic and
random dictionaries.  For the random case, width is measured by
\(M_{\mathrm{eff}}\) to account for logarithmic oversampling, thereby making
the theoretical convergence order directly visible.

\begin{figure}[htbp]
\centering
\includegraphics[width=\textwidth]{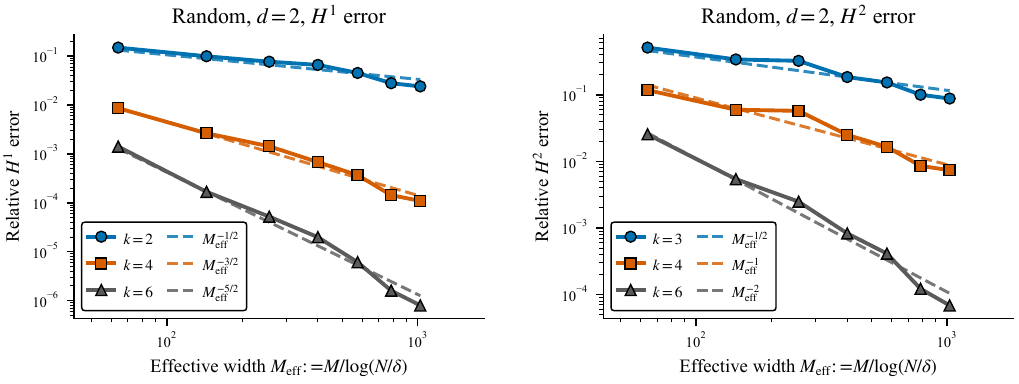}
\caption{Two-dimensional relative \(H^1\) (left) and \(H^2\) (right) errors
for random \(\erf\) dictionaries.  The coefficients are fitted only
in \(L^2\); dashed lines have order \((k-m)/2\).}
\label{fig:num-sobolev-2d}
\end{figure}
\FloatBarrier

The \(L^2\)-fitted approximants in Figure~\ref{fig:num-sobolev-2d}
display well-resolved \(H^1\) and \(H^2\) convergence orders
\((k-1)/2\) and \((k-2)/2\), respectively, without refitting in the
stronger norms, in agreement with the \(m=1,2\) cases of
Theorem~\ref{thm:random}.  In particular, the observed half-order loss
with each additional derivative directly resolves the dependence on \(m\)
predicted by \((k-m)/d\) in two dimensions.

Figure~\ref{fig:num-3d-10d} illustrates the role of dimension in the predicted
\(k/d\) rate through experiments with \(d=3\) and \(d=10\).  In both panels,
increasing \(k\) produces progressively steeper curves in accordance with the
regularity hierarchy.  The separation of the ten-dimensional curves shows that
the influence of \(k\) remains discernible even when the factor \(1/d\)
substantially moderates the convergence rate.

\begin{figure}[htbp]
\centering
\includegraphics[width=0.92\textwidth]{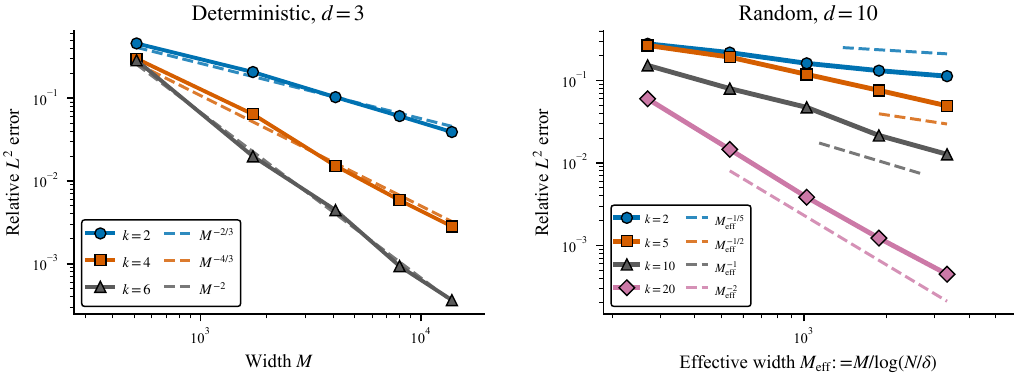}
\caption{Higher-dimensional relative \(L^2\) errors for deterministic \(\erf\)
dictionaries in \(d=3\) (left) and random \(\erf\) dictionaries in \(d=10\)
(right).  The right-panel curves and shaded regions are pointwise medians and
interquartile ranges over three target realizations on a common five-point
effective-width grid.  Dashed lines have order \(k/d\).}
\label{fig:num-3d-10d}
\end{figure}
\FloatBarrier

Across the three experiments, the parameters \(k\), \(m\), and \(d\) are
varied to isolate each dependence in the exponent \((k-m)/d\) predicted by
Theorems~\ref{thm:main} and~\ref{thm:random}.  The resulting convergence rates
for deterministic and random dictionaries closely track their theoretical
counterparts, providing consistent numerical evidence for the Sobolev
approximation hierarchy.

\section{Conclusion}
\label{sec:conclusion}

This work establishes regularity-dependent Sobolev approximation rates for
shallow networks generated by a fixed smooth sigmoidal activation.
Deterministic direction--offset dictionaries achieve the optimal Sobolev
order, while independent sampling from any parameter density bounded away
from zero preserves the same exponent with high probability, up to
logarithmic oversampling.  The numerical experiments also clearly recover
these rates and their predicted dependence on target
regularity, the Sobolev order of the error norm, and spatial dimension.

Natural extensions include broader regularity classes and adaptive feature
selection.  Extending the analysis beyond isotropic Sobolev spaces would
accommodate more general forms of regularity, while a connection with best
\(N\)-term approximation could allow the directions, offsets, and scales to
respond to the local structure of the target.  Such adaptivity is particularly
relevant for PDE solutions with interface or boundary-layer singularities, for
which uniform feature distributions may be inefficient.  Turning these
approximation results into convergence theory for PDE solvers will additionally
require stability, consistency, conditioning, and discretization estimates.

\section{Proofs and analysis}
\label{sec:proofs}

The proof begins with a bandlimited sigmoidal ridge representation and derives
the localization, quadrature, and parameter-space bandwidth estimates needed
for discrete approximation with polynomial coefficient control. Uniform offset
quadrature combined with spherical designs yields the deterministic
construction, while Gram-matrix concentration for i.i.d.\ samples in the joint
direction--offset space gives the random approximation result.

\subsection{Sigmoidal integral representation}
\label{sec:continuous-representation}

We first construct the continuous ridge representation shared by the
deterministic and random feature spaces, requiring only invertibility of the
activation derivative on the retained frequency band.

Fix \(d\geq2\), \(k\geq1\), and a bounded Lipschitz domain
\(\Omega\Subset\R^d\).  Put \(R:=\sup_{x\in\Omega}|x|\), fix
\(0<\kappa<\pi/2\), and choose fixed radii \(R<B_0<B\).

Let \(\phi:\R\to\R\) be a fixed smooth sigmoid, put
\(G=\phi'\in L^1(\R)\).
Fix an even \(\chi\in C_c^\infty(\R)\) with \(\chi=1\) on
\([-1,1]\) and \(\supp\chi\subset[-2,2]\).

\subsubsection{Fourier cutoff}

Choose $R_0>0$ so that $\overline\Omega\subset B_{R_0}$, and fix a
bounded extension/localization operator
\begin{equation}
 E:H^k(\Omega)\longrightarrow H^k(\R^d),
 \qquad \supp(Eu)\subset B_{R_0},
 \qquad \norm{Eu}_{H^k(\R^d)}\leq C\norm u_{H^k(\Omega)}.
 \tag{4.1}
\end{equation}
Such an operator is obtained by composing a standard Sobolev extension
with a fixed cutoff that equals one near $\overline\Omega$.  Put $f=Eu$.

Our Fourier convention is
\[
 \widehat f(\zeta):=\int_{\R^d}f(x)e^{-ix\cdot\zeta}\,dx,
 \qquad
 f(x)=(2\pi)^{-d}\int_{\R^d}\widehat f(\zeta)e^{ix\cdot\zeta}\,d\zeta.
\]
Define the radial low-pass function by
\(\widehat f_N(\zeta):=\chi(\abs\zeta/N)\widehat f(\zeta)\).

\begin{lemma}
\label{lem:cutoff}
For every \(N\geq1\) and every integer $m$ with $0\leq m\leq k$,
\begin{equation}
 \norm{f-f_N}_{H^m(\R^d)}
 \leq C N^{m-k}\norm f_{H^k(\R^d)}.
 \tag{4.2}
\end{equation}
\end{lemma}

\begin{proof}
Since $1-\chi(\abs\zeta/N)$ is bounded and vanishes for
$\abs\zeta\leq N$, Plancherel gives
\[
 \norm{f-f_N}_{H^m}^2
 \leq C N^{-2(k-m)}
       \int_{\abs\zeta>N}(1+\abs\zeta^2)^k
       \abs{\widehat f(\zeta)}^2\,d\zeta
 \leq C N^{-2(k-m)}\norm f_{H^k}^2.
\]
\end{proof}

\subsubsection{Fourier-slice coefficient density}

Let \(\mathcal R f(\omega,t):=\int_{x\cdot\omega=t}f(x)\,dS_x\) denote the Radon
transform over affine hyperplanes.  With the convention above, the
Fourier-slice identity reads
\(\mathcal F_t\mathcal R f(\omega,\tau)=\widehat f(\tau\omega)\).
Let $\Lambda^{d-1}$ denote the $t$-Fourier multiplier
$\abs\tau^{d-1}$, and set \(g_N:=\Lambda^{d-1}\mathcal R f_N\).
Set \(c_d=[2(2\pi)^{d-1}]^{-1}\), the full-sphere inversion constant
for our Fourier convention.
The standard Radon slice and inversion identities~\cite{Radon} give
\begin{equation}
 \begin{aligned}
 \widehat g_N(\omega,\tau)
 &=\abs\tau^{d-1}\widehat f_N(\tau\omega)
 =\abs\tau^{d-1}\chi(\abs\tau/N)\widehat f(\tau\omega),\\
 f_N(x)&=c_d\int_{\Sph^{d-1}}g_N(\omega,\omega\cdot x)\,d\omega.
 \end{aligned}
 \tag{4.3}
\end{equation}
In particular, \(g_N\) is exactly bandlimited in the ridge variable, with
\(\supp\widehat g_N(\omega,\cdot)\subset[-2N,2N]\).

\subsubsection{Deconvolution and synthesis}

Choose \(\sigma_N>0\) so that
\(\widehat G(\sigma_N\tau)\ne0\) on \(|\tau|\leq2N\), and define
\[
 \widehat q_N(\omega,\tau)
 :=\widehat G(\sigma_N\tau)^{-1}\widehat g_N(\omega,\tau),
 \qquad a_N:=\partial_tq_N.
\]
Define the classical entire function \(\Phi_1(z):=(e^z-1)/z\), with
\(\Phi_1(0)=1\), and the centered kernel
\(K_\sigma^\phi(t,b):=\phi((t-b)/\sigma)-\phi(-b/\sigma)\).
Then \((1-e^{-it\eta})/(i\eta)=t\Phi_1(-it\eta)\), including at
\(\eta=0\) by continuity.
The offset Fourier transform of the centered kernel is
\[
 \widehat K_\sigma^\phi(t,\eta)
 =t\Phi_1(-it\eta)\widehat G(-\sigma\eta).
\]
Since \(\widehat q_N\widehat G(\sigma_N\cdot)=\widehat g_N\), Fourier
inversion in the ridge variable followed by (4.3) gives
\begin{equation}
 f_N(x)-f_N(0)
 =c_d\int_{\Sph^{d-1}}\int_\R
 a_N(\omega,b)K_{\sigma_N}^\phi(\omega\cdot x,b)
 \,db\,d\omega.                                           \tag{4.4}
\end{equation}
Thus, with \(\mathcal A_N^\phi u:=a_N\), (4.4) gives the continuous
sigmoidal ridge representation used throughout the construction.

The centering is essential for the proof.  An uncentered sigmoid does not decay as
\(|b|\to\infty\), whereas the subtracted term in
\(K_{\sigma_N}^\phi\) is independent of \(x\) and can be absorbed into
the network output constant.  Subsection~\ref{sec:activation-setting}
specifies when (4.4) can be
truncated and discretized with superalgebraically small error.

\subsection{Standard sigmoid profiles and parameter localization}
\label{sec:activation-setting}

The representation in Subsection~\ref{sec:continuous-representation}
requires only that the deconvolution factor be well defined on the retained
frequency band.
Its deterministic and random discretizations additionally require
uniform control of reciprocal-lattice aliasing, physical offset tails,
bias tails, and angular polynomial tails, together
with polynomial local growth of the required derivatives.  We collect
these requirements here so that the remainder of the proof applies
uniformly to the three standard sigmoids through their two common
derivative profiles.

For the activation \(\phi\), derivative \(G\), and centered kernel
introduced in Subsection~\ref{sec:continuous-representation}, write
\begin{equation}
 G_\sigma(t):=\sigma^{-1}G(t/\sigma),\qquad
 K_\sigma^\phi(t,b)
 :=\phi((t-b)/\sigma)-\phi(-b/\sigma).
 \tag{4.5}
\end{equation}
Then
\begin{equation}
 \widehat G_\sigma(\tau)=\widehat G(\sigma\tau),\qquad
 \widehat K_\sigma^\phi(t,\eta)
 =t\Phi_1(-it\eta)\widehat G(-\sigma\eta).
 \tag{4.6}
\end{equation}

\begin{assumption}[Standard sigmoidal profiles]
\label{ass:activation}
The activation \(\phi:\R\to\R\) is a bounded \(C^\infty\) sigmoid whose
derivative is one of the normalized profiles
\begin{equation}
 G(t)=\phi'(t)=e^{-t^2}
 \qquad\text{or}\qquad
 G(t)=\phi'(t)=\operatorname{sech}^2t.                   \tag{4.7}
\end{equation}
Throughout, a fixed replacement
\(\widetilde\phi(t)=A_{\rm aff}\phi(\lambda t)+B_{\rm aff}\), with
\(A_{\rm aff}\ne0\), \(\lambda>0\), and \(B_{\rm aff}\in\R\), is regarded as
the same profile: \(A_{\rm aff},B_{\rm aff}\) are absorbed by the outer layer,
and \(\lambda\) by the
common inner scale.  Hence (4.7) covers the usual \(\tanh\), logistic
sigmoid, and \(\erf\) after fixed input--output normalization.
\end{assumption}

Put \(Q_*:=k+d+1\), \(h_N:=\kappa/N\), and
\(\lambda_n:=\pi n/B\), and define the
in-band deconvolution factor
\(D_N^\phi:=\sup_{|\tau|\leq2N}|\widehat G(\sigma_N\tau)|^{-1}\).
Let \(\mathcal P_L^{\Sph}\) be the spherical polynomials of degree at
most \(\lfloor L\rfloor\), and write
\[
 E_L^{\Sph}(U):=\inf_{P\in\mathcal P_L^{\Sph}}
 \|U-P\|_{L^\infty(\Sph^{d-1})}.
\]

Set \(\mathbb D_N^\phi:=D_N^\phi N^{d+1}\) and, in the bias term,
\(s_{n,\tau}:=\tau-\lambda_n\).  For \(0\leq j\leq k\), define
the activation-dependent errors
\begin{gather}
 \varepsilon_{N,j}^{\rm ali}
 :=\mathbb D_N^\phi
 \sum_{m\ne0}\sup_{|\tau|\leq2N}
 \left(1+\left|\tau+\frac{2\pi m}{h_N}\right|\right)^j
 \left|\widehat G\!\left(\sigma_N
 \left(\tau+\frac{2\pi m}{h_N}\right)\right)\right|,
                                                               \tag{4.8}\\
 \varepsilon_{N,j}^{\rm tail}
 :=\mathbb D_N^\phi
 \biggl[
 \int_{|b|>B_0}\sup_{|t|\leq R}
 |\partial_t^jK_{\sigma_N}^\phi(t,b)|\,db 
 +h_N\sum_{|\ell h_N|>B_0}\sup_{|t|\leq R}
 |\partial_t^jK_{\sigma_N}^\phi(t,\ell h_N)|
 \biggr],                                                     \tag{4.9}\\
 \varepsilon_{N,j}^{\rm bias}(C)
 :=\mathbb D_N^\phi
 \sum_{|n|>CN}\sup_{\substack{|\tau|\leq2N\\|t|\leq R}}
 (1+|s_{n,\tau}|)^j
 \left|t\Phi_1(its_{n,\tau})
 \widehat G(-\sigma_Ns_{n,\tau})\right|,                  \tag{4.10}\\
 \varepsilon_{N,j}^{\rm ang}(C)
 :=\mathbb D_N^\phi\max_{|\alpha|\leq j}
 \sup_{\substack{x\in\Omega\\|b|\leq B}}
 E_{CN}^{\Sph}\!\left(
 \partial_x^\alpha K_{\sigma_N}^\phi(\omega\cdot x,b)
 \right).                                                     \tag{4.11}
\end{gather}

\begin{lemma}[Localization estimates for the standard sigmoids]
\label{lem:standard-sigmoids}
Fix \(k\geq1\) and \(Q>0\), and suppose that
Assumption~\ref{ass:activation} holds.  There is a constant \(A_\phi>0\),
depending only on \(d,k,Q,R,B_0,B,\kappa\), such that, for \(N\geq2\),
the choice
\[
 \sigma_N=
 \begin{cases}
  \sqrt{A_\phi\log N}/N,&G(t)=e^{-t^2},\\
  A_\phi\log N/N,&G(t)=\operatorname{sech}^2t,
 \end{cases}
\]
has the following properties.
There are \(N_{\phi,k,Q}\geq2\), \(C_{\phi,k,Q}>0\),
\(S_{\phi,k}\geq0\), and \(C_{\phi,k,Q}^{\rm bw}>0\), depending only
on the fixed data and $Q$ and independent of \(N\), such that, for every
\(N\geq N_{\phi,k,Q}\),
\begin{enumerate}
 \item The profile is nondegenerate on the retained band and satisfies
 the quantitative bounds
 \begin{equation}
 \begin{gathered}
  \widehat G(\sigma_N\tau)\ne0\quad(\tau\in\R),
  \qquad D_N^\phi\leq C_{\phi,k,Q}N^{P_{\rm inv}},\\
  \max_{0\leq j\leq k}\sup_{|t|\leq R,\,|b|\leq B}
  |\partial_t^jK_{\sigma_N}^\phi(t,b)|
  \leq C_{\phi,k,Q}N^{S_{\phi,k}},
 \end{gathered}
 \tag{4.12}\label{eq:profile-quantitative-bounds}
 \end{equation}
 where
 \[
  P_{\rm inv}=
  \begin{cases}
   A_\phi,&G(t)=e^{-t^2},\\
  \pi A_\phi,&G(t)=\operatorname{sech}^2t,
  \end{cases}.
 \]
 \item For every \(0\leq j\leq k\),
 \begin{equation}
  \varepsilon_{N,j}^{\rm ali}
  +\varepsilon_{N,j}^{\rm tail}
  \leq C_{\phi,k,Q}N^{-Q}.
  \tag{4.13}\label{eq:profile-offset-errors}
 \end{equation}
 \item For every \(0\leq j\leq k\),
 \begin{equation}
  \varepsilon_{N,j}^{\rm bias}(C_{\phi,k,Q}^{\rm bw})
  +\varepsilon_{N,j}^{\rm ang}(C_{\phi,k,Q}^{\rm bw})
  \leq C_{\phi,k,Q}N^{-Q}.
  \tag{4.14}\label{eq:profile-bandwidth-errors}
 \end{equation}
\end{enumerate}
In the main theorems we use the single order \(Q_*=k+d+1\).  The
corresponding constants and exponents are then fixed independently of
\(N\).
\end{lemma}

\begin{proof}
First consider the Gaussian profile \(G(t)=e^{-t^2}\).  Then
\(\widehat G(\tau)=\sqrt{\pi}e^{-\tau^2/4}\) under the fixed
Fourier convention.  For \(N\geq2\), take
\begin{equation}
 \sigma_N=\frac{\sqrt{A_\phi\log N}}{N},
 \qquad A_\phi>0,\qquad 0<\kappa<\frac{\pi}{2}.          \tag{4.15}
\end{equation}
Then
\(D_N^\phi\leq CN^{A_\phi}\), and the
local derivative bound in part~1 of the lemma follows from the
Hermite--Gaussian formulas; thus \(P_{\rm inv}=A_\phi\).  The displayed
Fourier transform is strictly positive on \(\R\), and the same
Hermite--Gaussian formulas give \(G\in\mathcal S(\R)\).  Hence the
scale choice (4.15) proves \eqref{eq:profile-quantitative-bounds} for
the Gaussian profile.  Put
\(c_\kappa=\pi/\kappa-1>1\).  Since
\(
 |\tau+2\pi m/h_N|\geq2N(\pi|m|/\kappa-1)
\)
for \(|\tau|\leq2N\), Gaussian decay gives, uniformly for \(j\leq k\),
\[
 \varepsilon_{N,j}^{\rm ali}
 \leq C N^{d+1+j+A_\phi}\sum_{m\ne0}(1+|m|)^j
 N^{-A_\phi(\pi|m|/\kappa-1)^2}.
\]
Thus, for any prescribed \(Q>0\), (4.8) is \(O(N^{-Q})\) whenever
\begin{equation}
 A_\phi(c_\kappa^2-1)>Q+d+k+2.                           \tag{4.16}
\end{equation}
For the physical tail, write
\(K_{\sigma_N}^\phi(t,b)=\int_0^tG_{\sigma_N}(s-b)\,ds\); for \(j\geq1\),
\(\partial_t^jK_{\sigma_N}^\phi(t,b)
=\sigma_N^{-j}G^{(j-1)}((t-b)/\sigma_N)\).
The Hermite--Gaussian formulas, \(B_0>R\), Gaussian tail integration,
and comparison of the offset sum with the corresponding integral give,
uniformly for \(j\leq k\),
\[
 \varepsilon_{N,j}^{\rm tail}
 \leq C N^{A_\phi+d+1}\sigma_N^{1-k}
 e^{-c(B_0-R)^2/\sigma_N^2}
 =O(N^{-Q})
\]
for every prescribed \(Q>0\).  Combining this tail bound with the preceding
aliasing estimate and (4.16) proves
\eqref{eq:profile-offset-errors} for the Gaussian profile.

We next prove \eqref{eq:profile-bandwidth-errors}.  For
\(C_{\rm b}\geq\max\{1,4B/\pi\}\), \(|n|>C_{\rm b}N\), and
\(|\tau|\leq2N\), one has
\(|s_{n,\tau}|\geq\pi|n|/B-2N\geq\pi|n|/(2B)\).
Moreover, for \(|t|\leq R\),
\(\left|t\Phi_1(its_{n,\tau})\right|
=\left|(e^{its_{n,\tau}}-1)/(is_{n,\tau})\right|
\leq |t|\leq R\).
Hence (4.10), \(D_N^\phi\leq CN^{A_\phi}\), and the Gaussian Fourier
transform give, uniformly for \(j\leq k\),
\[
\begin{aligned}
 \varepsilon_{N,j}^{\rm bias}(C_{\rm b})
 &\leq C N^{A_\phi+d+1}
 \sum_{|n|>C_{\rm b}N}|n|^j
 \exp\!\left(-\frac{\pi^2A_\phi\log N}{16B^2}
                  \frac{n^2}{N^2}\right)\\
 &\leq C
 N^{A_\phi+d+k+2-\pi^2A_\phi C_{\rm b}^2/(32B^2)}.
\end{aligned}
\]
In the last step, half of the Gaussian decay at
\(|n|=C_{\rm b}N\) is used to control the tail, and the remaining sum is
bounded by \(CN^{j+1}\).  Hence \(C_{\rm b}\) can be fixed so that
\[
 \frac{\pi^2A_\phi C_{\rm b}^2}{32B^2}
 >Q+A_\phi+d+k+2.
\]

Then we give the details of the angular estimate.  Fix \(x\in\Omega\), write
\(r=|x|\), and rotate coordinates so that
\(\omega\cdot x=rz\), where \(z\in[-1,1]\).  If \(|\alpha|=j\),
\[
 \partial_x^\alpha K_{\sigma_N}^\phi(\omega\cdot x,b)
 =\omega^\alpha\partial_t^jK_{\sigma_N}^\phi(rz,b),
 \qquad
 \partial_t^jK_{\sigma_N}^\phi(t,b)
 =\sigma_N^{-j}G^{(j-1)}((t-b)/\sigma_N)
\]
for \(j\geq1\); the case \(j=0\) is the corresponding primitive
difference.  The factor \(\omega^\alpha\) is a spherical polynomial of
degree at most \(j\), so it suffices to approximate the remaining zonal
factor with degree \(L-j\).

Let \(\mathcal E_{e^\theta}\) be the Bernstein ellipse in the \(z\)-plane.
The Chebyshev projection estimate
\cite[Theorem~8.2]{TrefethenATAP} shows that, if \(F\) is analytic on
\(\mathcal E_{e^\theta}\) and bounded there by \(M_\theta\), then there
is a univariate polynomial \(p_m\) of degree at most \(m\) such that
\[
 \|F-p_m\|_{L^\infty([-1,1])}
 \leq \frac{2M_\theta e^{-m\theta}}{e^\theta-1}
 \leq C\theta^{-1}e^{-m\theta}M_\theta,
 \qquad 0<\theta\leq1.
\]
Since \(z\) is a linear coordinate of \(\omega\), \(p_m(z)\) is a
spherical polynomial of degree at most \(m\).  For \(L\geq k\), taking
\(m=\lfloor L\rfloor-j\) and multiplying by \(\omega^\alpha\) therefore
gives a spherical polynomial of degree at most \(\lfloor L\rfloor\).
This estimate is valid in every dimension \(d\geq2\).  On
\(\mathcal E_{e^\theta}\),
\(|\operatorname{Im}z|\leq\sinh\theta\).  The Hermite--Gaussian formula
for \(G^{(q)}\) and integration along a straight segment when \(j=0\)
therefore give, uniformly for \(|b|\leq B\),
\[
 \max_{z\in\mathcal E_{e^\theta}}
 |\partial_t^jK_{\sigma_N}^\phi(rz,b)|
 \leq C N^{\Pi_0}
 \exp\!\left(Cr^2\sinh^2\theta/\sigma_N^2\right),
 \qquad 0\leq j\leq k,
\]
where \(\Pi_0\) is independent of \(L,N,r,b\).  Take
\(\theta=c_1L\sigma_N^2/(r^2+L\sigma_N^2)\) with a sufficiently small
fixed \(c_1\); for the values \(L\asymp N\) used below, \(\theta\leq1\)
when \(N\) is large.  Moreover, \(\theta^{-1}\leq CN\) uniformly for
\(0<r\leq R\) and can be absorbed into the polynomial prefactor.
Substitution in the preceding bound, with \(r=0\) treated separately,
yields uniformly for \(0\leq r\leq R\)
\[
 \max_{|\alpha|\leq k}\sup_{x,b}
 E_L^{\Sph}\!\left(\partial_x^\alpha
 K_{\sigma_N}^\phi(\omega\cdot x,b)\right)
 \leq C N^{\Pi_{\rm G}}
 e^{-cL^2\sigma_N^2/(R^2+L\sigma_N^2)}.
\]
Here \(\Pi_{\rm G}\) is finite, depends only on
\(d,k,A_\phi,R,B\), and is independent of \(L\) and \(N\).
At \(L=C_{\rm ang}N\), first fixing \(A_\phi\) by (4.16) and then fixing
\(C_{\rm ang}\) so that
\(cA_\phi C_{\rm ang}^2/(1+R^2)>Q+\Pi_{\rm G} + A_\phi + d + 1\) makes the angular
indicator in (4.11) \(O(N^{-Q})\).  Combining this angular estimate with
the preceding bias bound and taking
\(C_{\phi,k,Q}^{\rm bw}:=\max(C_{\rm b},C_{\rm ang})\)
proves \eqref{eq:profile-bandwidth-errors} for the Gaussian profile.

Next we consider \(G(t)=\operatorname{sech}^2t\) and use the normalized
representative
\begin{equation}
 \phi(t)=\tanh t,\qquad G(t)=\operatorname{sech}^2t,\qquad
 \widehat G_\sigma(\tau)
 =\frac{\pi\sigma\tau}{\sinh(\pi\sigma\tau/2)}.         \tag{4.17}
\end{equation}
The Fourier transform does not vanish on \(\R\).  With \(A_\phi>0\) and
\(\sigma_N=A_\phi\log N/N\) for \(N\geq2\), we
have \(D_N^\phi\leq CN^{\pi A_\phi}\), while
all local derivatives grow at most polynomially; thus
\(P_{\rm inv}=\pi A_\phi\).  Exponential decay of all derivatives of
\(G\) also gives \(G\in\mathcal S(\R)\).  Consequently, (4.17) and the
chosen scale prove \eqref{eq:profile-quantitative-bounds} for the
\(\operatorname{sech}^2\) profile.  The estimate
\(|\widehat G(\xi)|\leq C(1+|\xi|)e^{-\pi|\xi|/2}\), the preceding
reciprocal-lattice gap, and \(\kappa<\pi/2\) give
\[
 \varepsilon_{N,j}^{\rm ali}
 \leq C N^{d+k+2+\pi A_\phi}
 \sum_{m\ne0}(1+|m|)^{k+1}
 N^{-\pi A_\phi(\pi|m|/\kappa-1)}=O(N^{-Q})
\]
for any prescribed \(Q>0\), provided
\begin{equation}
 \pi A_\phi\left(\frac{\pi}{\kappa}-2\right)>Q+d+k+2.    \tag{4.18}
\end{equation}
By the same argument, exponential decay of all derivatives of \(G\)
gives, uniformly for \(j\leq k\),
\[
 \varepsilon_{N,j}^{\rm tail}
 \leq C N^{\pi A_\phi+d+1}\sigma_N^{1-k}
 e^{-c(B_0-R)/\sigma_N}
 =O(N^{-Q}).
\]
The preceding aliasing estimate, this tail bound, and (4.18) prove
\eqref{eq:profile-offset-errors} for the \(\operatorname{sech}^2\) profile.

We next prove \eqref{eq:profile-bandwidth-errors}.  The same Fourier
estimate and the gap
\(|s_{n,\tau}|\geq(\pi C_{\rm b}/B-2)N\) give
\[
 \varepsilon_{N,j}^{\rm bias}(C_{\rm b})
 \leq C N^{d+k+2+\pi A_\phi-c_{\rm b}A_\phi C_{\rm b}},
\]
where
\(c_{\rm b}>0\) depends only on \(B\) and the fixed Fourier
normalization.

Finally, \(\tanh z\) is holomorphic in
\(|\operatorname{Im}z|<\pi/2\).  The same reduction
\(\omega\cdot x=rz\) and the same Chebyshev projection estimate apply.
Choose \(0<c_0<\pi/2\) and, for \(r>0\), set
\(\theta=\operatorname{arsinh}(\min\{1,c_0\sigma_N/r\})\).  Then
\(|\operatorname{Im}((rz-b)/\sigma_N)|\leq c_0\), so the ellipse remains
in a fixed smaller holomorphy strip.  Cauchy's estimate gives, uniformly
for \(j\leq k\),
\[
 \max_{z\in\mathcal E_{e^\theta}}
 |\partial_t^jK_{\sigma_N}^\phi(rz,b)|
 \leq C\sigma_N^{-j}.
\]
The case \(r=0\) is immediate, and multiplication by \(\omega^\alpha\)
only shifts the spherical degree by at most \(k\).  Moreover,
\(\theta^{-1}\leq C(1+r/\sigma_N)\leq CN\), so the Chebyshev prefactor
is absorbed into the polynomial factor.  Consequently the analytic
zonal estimate yields
\(
 \varepsilon_{N,j}^{\rm ang}(C_{\rm ang})
 \leq C N^{d+k+2+\pi A_\phi-c_{\rm ang}A_\phi C_{\rm ang}/R}.
\)
Here \(c_{\rm ang}>0\) depends only on the fixed smaller holomorphy
strip.  After \(A_\phi\) is fixed by (4.18), choose \(C_{\rm b}\) and
\(C_{\rm ang}\) so that the last two exponents are below \(-Q\), and
set
\(C_{\phi,k,Q}^{\rm bw}=\max(C_{\rm b},C_{\rm ang})\).
The bias and angular bounds then prove
\eqref{eq:profile-bandwidth-errors} for every $j\leq k$.  The standard logistic
sigmoid is \((1+\tanh(t/2))/2\), so it is covered by the fixed
input--output normalization stated in Assumption~\ref{ass:activation}.
This completes both profile
cases and proves the lemma.
\end{proof}

\begin{remark}
\label{rem:broader-activations}
In general, the analysis extends to any smooth sigmoid
\(\phi\) with \(G=\phi'\in\mathcal S(\R)\) and zero-free \(\widehat G\),
provided that, for every prescribed \(J>0\), \(\sigma_N\) can be chosen
so that \(\sigma_N^{-1}\) and \(D_N^\phi\) grow at most polynomially and
(4.8)--(4.11) are \(O(N^{-J})\).
\end{remark}

\subsection{Offset quadrature}
\label{sec:offset-quadrature}

We now discretize the offset variable.  The first lemma justifies the
Poisson formula used by the uniform rule; the second collects the resulting
coefficient, truncation, and quadrature estimates.

Set \(b_{\ell,N}=\ell h_N\), where \(h_N=\kappa/N\), and define
\begin{equation}
 \begin{aligned}
 F_N^\phi(\omega,x)&:=h_N
 \sum_{|\ell h_N|\leq B}a_N(\omega,\ell h_N)
 \phi\!\left(\frac{\omega\cdot x-\ell h_N}{\sigma_N}\right),\\
 \overline F_N^\phi(\omega,x)&:=F_N^\phi(\omega,x)-F_N^\phi(\omega,0).
 \end{aligned}
 \tag{4.19}
\end{equation}

\begin{lemma}[Schwartz admissibility for offset sampling]
\label{lem:poisson-admissibility}
Under the hypotheses and scale choices of
Lemma~\ref{lem:standard-sigmoids}, let $f\in H^k(\R^d)$ have fixed
compact support and let $a_N$ be the coefficient density constructed
in Subsection~\ref{sec:continuous-representation}.  For every
$N\geq N_{\phi,k,Q}$,
$\omega\in\Sph^{d-1}$,
$|t|\leq R$, and integer $0\leq j\leq k$, the function
\begin{equation}
 F_{N,\omega,t,j}(b)
 :=a_N(\omega,b)\partial_t^jK_{\sigma_N}^\phi(t,b)
 \tag{4.20}\label{eq:poisson-function}
\end{equation}
belongs to $\mathcal S(\R)$.  Its Schwartz seminorms are uniform in
$\omega$ and $|t|\leq R$ and grow at most polynomially in $N$.
Consequently, for every $h>0$,
\begin{equation}
 h\sum_{\ell\in\mathbb Z}F_{N,\omega,t,j}(\ell h)
 =\sum_{m\in\mathbb Z}
 \widehat F_{N,\omega,t,j}\!\left(\frac{2\pi m}{h}\right),
 \tag{4.21}\label{eq:poisson-formula}
\end{equation}
and both series converge absolutely.  Moreover,
\begin{equation}
 |\widehat F_{N,\omega,t,j}(\eta)|
 \leq C_R\int_{-2N}^{2N}|\widehat a_N(\omega,\tau)|
 (1+|\eta-\tau|)^j
 |\widehat G(\sigma_N(\eta-\tau))|\,d\tau .
 \tag{4.22}\label{eq:poisson-alias-bound}
\end{equation}
\end{lemma}

\begin{proof}
The compact frequency support of $\widehat a_N$ and the fixed support of
$f$ imply, for every integer $r\geq0$,
\begin{equation}
 \sup_{\omega\in\Sph^{d-1}}
 \int_{-2N}^{2N}|\tau|^r|\widehat a_N(\omega,\tau)|\,d\tau
 \leq C_rD_N^\phi N^{d+r+1}\|f\|_{H^k(\R^d)}.
 \tag{4.23}\label{eq:a-derivative-bound}
\end{equation}
Indeed, $\|\widehat f\|_\infty\leq\|f\|_1
\leq C\|f\|_{L^2}$ on the fixed support.  Fourier inversion therefore
shows that every $b$-derivative of $a_N(\omega,\cdot)$ is bounded,
uniformly in $\omega$, by a fixed power of $N$.

By part~1 of Lemma~\ref{lem:standard-sigmoids}, $G$ and all its
derivatives are Schwartz functions.  Since
\(K_{\sigma_N}^\phi(t,b)=\int_0^tG_{\sigma_N}(s-b)\,ds\) and
\(\partial_t^jK_{\sigma_N}^\phi(t,b)
=\partial_t^{j-1}G_{\sigma_N}(t-b)\) for \(j\geq1\),
all $b$-Schwartz seminorms of $\partial_t^jK_{\sigma_N}^\phi(t,\cdot)$
are uniform for $|t|\leq R$ and grow at most polynomially in $N$.
Leibniz' rule and \eqref{eq:a-derivative-bound} now give
$F_{N,\omega,t,j}\in\mathcal S(\R)$ with the asserted uniformity.
The classical Poisson formula for Schwartz functions gives
\eqref{eq:poisson-formula}.  Finally, the product--convolution identity
and (4.6) give
\[
 \widehat F_{N,\omega,t,j}(\eta)
 =\frac1{2\pi}\int_{-2N}^{2N}\widehat a_N(\omega,\tau)
 \partial_t^j\widehat K_{\sigma_N}^\phi(t,\eta-\tau)\,d\tau .
\]
Using $|t\Phi_1(-it\xi)|\leq|t|$ for $j=0$ and the explicit
$t$-derivatives of (4.6) for $j\geq1$ proves
\eqref{eq:poisson-alias-bound}.
\end{proof}

Retain the cutoff, coefficient density \(a_N\), and inversion constant
constructed in Subsection~\ref{sec:continuous-representation}.

\begin{lemma}[Coefficient and offset quadrature estimates]
\label{lem:offset-estimates}
Fix \(k\geq1\) and \(Q>0\), let \(\phi\) satisfy
Assumption~\ref{ass:activation}, and choose the scale and constants as in
Lemma~\ref{lem:standard-sigmoids}.  Let \(f\in H^k(\R^d)\) be supported in
a fixed ball and set
\[
 \gamma_{\phi,d,k}:=P_{\rm inv}
 +\left(\frac{d+1}{2}-k\right)_+.
\]
Then, for every \(N\geq N_{\phi,k,Q}\), the following conclusions hold.
\begin{enumerate}
\item With \(\beta_{d,k}:=(d-k+\tfrac12)_+\),
\begin{equation}
 \begin{aligned}
 \sup_\omega\|\widehat a_N(\omega,\cdot)\|_{L^1(\R)}
 &\leq C N^{P_{\rm inv}+\beta_{d,k}}\|f\|_{H^k(\R^d)},\\
 \|a_N\|_{L^2(\Sph^{d-1}\times[-B,B])}
 &\leq C N^{\gamma_{\phi,d,k}}\|f\|_{H^k(\R^d)}.
 \end{aligned}                                             \tag{4.24}
\end{equation}
\item For the fixed accuracy order \(Q\),
denote by
\({\cal E}_{N,j}^\phi(\omega,t)\) the quantity
\[
 \partial_t^j\!\left[
 \int_\R a_N(\omega,b)K_{\sigma_N}^\phi(t,b)\,db
 -h_N\!\sum_{|\ell h_N|\leq B}
 a_N(\omega,\ell h_N)K_{\sigma_N}^\phi(t,\ell h_N)
 \right].
\]
Then
\begin{equation}
 \max_{0\leq j\leq k}\sup_{\omega,\,|t|\leq R}
 |{\cal E}_{N,j}^\phi(\omega,t)|
 \leq C N^{-Q}\|f\|_{H^k(\R^d)}.                       \tag{4.25}
\end{equation}
\item For every multi-index $\alpha$ with $|\alpha|\leq k$, the
continuous and discrete physical tails satisfy the explicit bound
\begin{equation}
 \begin{aligned}
 \sup_{\substack{\omega\in\Sph^{d-1}\\x\in\Omega}}
 \bigg[&\int_{|b|>B_0}
  |a_N(\omega,b)\partial_x^\alpha
    K_{\sigma_N}^\phi(\omega\cdot x,b)|\,db\\
 &+h_N\sum_{|\ell h_N|>B_0}
  |a_N(\omega,\ell h_N)\partial_x^\alpha
    K_{\sigma_N}^\phi(\omega\cdot x,\ell h_N)|\bigg]
 \leq CN^{-Q}\|f\|_{H^k(\R^d)}.
 \end{aligned}                                           \tag{4.26}
\end{equation}
\end{enumerate}
\end{lemma}

\begin{proof}
The first estimate in (4.24) follows from
Cauchy--Schwarz in \(\tau\) and the uniform slice estimate
\[
 \sup_{\omega\in\Sph^{d-1}}
 \int_\R(1+\tau^2)^k|\widehat f(\tau\omega)|^2\,d\tau
 \leq C\|f\|_{H^k(\R^d)}^2.
\]
The latter follows by applying one-dimensional Plancherel to the Radon
transforms of \((\omega\cdot\nabla)^j f\), \(0\leq j\leq k\), and using
the fixed support of \(f\).  The second estimate in (4.24) follows from
Plancherel, polar integration, and
\[
 \sup_{0\leq r\leq2N}\frac{r^{d+1}}{(1+r^2)^k}
 \leq C N^{2((d+1)/2-k)_+}.
\]
For (4.25), apply Lemma~\ref{lem:poisson-admissibility} with $h=h_N$.
The zero Fourier sample in \eqref{eq:poisson-formula} is the continuous integral,
and \eqref{eq:poisson-alias-bound} bounds the sum of all nonzero reciprocal-lattice
samples by (4.8).  Removing the nodes and integral outside $[-B,B]$
costs no more than the stronger tail bound (4.9), which starts at
\(B_0<B\).  Finally,
$\partial_x^\alpha K_{\sigma_N}^\phi(\omega\cdot x,b)
=\omega^\alpha\partial_t^{|\alpha|}K_{\sigma_N}^\phi(\omega\cdot x,b)$,
and $\|a_N\|_\infty\leq(2\pi)^{-1}\|\widehat a_N\|_1$.
Thus (4.9) and (4.24) give (4.26).  Since
$\beta_{d,k}\leq d+1$, the $N^{d+1}$ factor in
$\mathbb D_N^\phi$ dominates the coefficient loss.  Part~2 of
Lemma~\ref{lem:standard-sigmoids} proves these conclusions.
\end{proof}

\begin{corollary}
\label{cor:offset-quadrature}
Under the hypotheses of Lemma~\ref{lem:offset-estimates}, simultaneously
for every integer \(0\leq m\leq k\),
\begin{equation}
 \left\|f_N-f_N(0)
 -c_d\int_{\Sph^{d-1}}\overline F_N^\phi(\omega,\cdot)\,d\omega
 \right\|_{H^m(\Omega)}
 \leq C N^{-Q}\|f\|_{H^k(\R^d)}.                        \tag{4.27}
\end{equation}
\end{corollary}

\begin{proof}
Differentiate the exact representation (4.4), apply (4.25) with
$j=0,\ldots,m$, and integrate in $\omega$.
\end{proof}

\subsection{Effective parameter-space bandwidth}
\label{sec:effective-bandwidth}

The offset quadrature error is already controlled by
Subsection~\ref{sec:offset-quadrature}.  We now
establish the angular bandwidth required by the spherical design.  The joint
estimate below also records the trigonometric bandwidth in the offset
parameter.  This distinct finite-dimensional estimate is used only for the
random parameter-space cubature in
Subsection~\ref{sec:parameter-bandwidth}.

Fix once and for all
\(\psi_c\in C_c^\infty((-B,B))\) with \(\psi_c=1\) on a neighborhood
of \([-B_0,B_0]\), and let \(\psi\) be its smooth \(2B\)-periodic
extension.  Write \(\Sph_B^1:=\R/(2B\mathbb Z)\) and
\(\Theta:=\Sph^{d-1}\times\Sph_B^1\); in periodic arguments,
\([-B,B]\) is identified with this circle up to its measure-zero
endpoints.  Let \({\cal T}_L^B\) be the trigonometric polynomials with
frequencies \(|n|\leq L\), and set
\({\cal V}_L:=\mathcal P_L^{\Sph}\otimes{\cal T}_L^B\) and
\(D_L:=\dim{\cal V}_L\asymp L^d\).

\begin{lemma}[Effective parameter bandwidth]
\label{lem:activation-bandwidth}
Under the hypotheses of Lemma~\ref{lem:offset-estimates}, put
\(\widetilde a_N=\psi a_N\) on \(\Theta\) and
\[
 H_N((\omega,b),x):=
 c_d a_N(\omega,b)K_{\sigma_N}^\phi(\omega\cdot x,b).
\]
For a $2B$-periodic function $T$ in the offset variable, use the
normalization
\[
 \widehat T_n(\omega,x):=\frac1{2B}\int_{-B}^B
 T(\omega,b,x)e^{-i\lambda_nb}\,db,
 \qquad \lambda_n=\frac{\pi n}{B}.
\]
There is a fixed \(C_*>0\), depending only on the fixed data and $Q$,
such that, for every
\(N\geq N_{\phi,k,Q}\), there are functions
\(\widetilde a_N^\sharp\in{\cal V}_{C_*N}\) and
\(H_N^\sharp\in{\cal V}_{C_*N}\otimes H^k(\Omega)\) for which,
simultaneously for every integer \(0\leq s\leq k\),
\begin{equation}
 \begin{aligned}
 &\|\widetilde a_N-\widetilde a_N^\sharp\|_{L^\infty(\Theta)}
 +\sup_{|b|\leq B}E_{C_*N}^{\Sph}(a_N(\cdot,b))\\
 &\qquad
 +\sup_{p\in\Theta}
 \|\psi H_N(p,\cdot)-H_N^\sharp(p,\cdot)\|_{H^s(\Omega)}
 \leq CN^{-Q}\|f\|_{H^k(\R^d)}.
 \end{aligned}
 \tag{4.28}
\end{equation}
In particular, the two product tails used in this conclusion have the
precise form
\begin{equation}
 \begin{aligned}
 \max_{|\alpha|\leq k}\bigg[&
 \sup_{\substack{\omega\in\Sph^{d-1}\\x\in\Omega}}
 \sum_{|n|>C_*N}
 \big|\widehat{\psi\,\partial_x^\alpha H_N}_n(\omega,x)\big|\\
 &+\sup_{\substack{|b|\leq B\\x\in\Omega}}
 E_{C_*N}^{\Sph}\!\left(
 \psi(b)\partial_x^\alpha H_N(\cdot,b,x)\right)\bigg]
 \leq CN^{-Q}\|f\|_{H^k(\R^d)}.
 \end{aligned}                                           \tag{4.29}
\end{equation}
\end{lemma}

\begin{proof}
We first record the spectral cutoffs used below.  Fix an even real
\(\vartheta\in C_c^\infty(\R)\), equal to one on \([-1,1]\) and
supported in \([-2,2]\), and define
\[
 {\cal Q}_L^bT=\sum_n\vartheta(n/L)\widehat T_ne^{i\pi nb/B},
 \qquad
 {\cal Q}_L^\omega U=\sum_{\ell\geq0}\vartheta(\ell/L)
 \operatorname{proj}_\ell U.
\]
Both operators reproduce degrees at most \(L\), vanish above \(2L\),
and are uniformly bounded on \(L^\infty\).  For \({\cal Q}_L^b\) this
follows directly from Poisson summation for its periodic convolution
kernel; for \({\cal Q}_L^\omega\) it follows from the localized spherical
kernel estimate~\cite[Theorem~3.5]{NPW} (and from the same periodic
argument when \(d=2\)).  They act in different parameter variables and
commute with physical derivatives.

We make the offset-frequency estimates explicit.  Set
$P_a=P_{\rm inv}+\beta_{d,k}$.  By (4.24),
\begin{equation}
 \sup_\omega\|\widehat a_N(\omega,\cdot)\|_1
 \leq CN^{P_a}\|f\|_{H^k}.                              \tag{4.30}
\end{equation}
Because $\psi_c$ is supported in $(-B,B)$, the periodic coefficient of
$\widetilde a_N=\psi a_N$ is the sampled Euclidean transform
\begin{equation}
 \widehat{\widetilde a_N}_n(\omega)
 =\frac1{2B}{\cal F}_b(\psi_ca_N)(\omega,\lambda_n)
 =\frac1{4\pi B}\int_{-2N}^{2N}
 \widehat\psi_c(\lambda_n-\tau)\widehat a_N(\omega,\tau)\,d\tau.
 \tag{4.31}
\end{equation}
Choose $C_{\rm off}>2B/\pi$.  If $|n|>C_{\rm off}N$ and $|\tau|\leq2N$, then
\(|\lambda_n-\tau|\geq
(\pi/B-2/C_{\rm off})|n|=:c_{\rm gap}|n|\).
Since $\widehat\psi_c$ is Schwartz, (4.30)--(4.31) imply, for every
$J>1$,
\begin{equation}
 \sup_\omega\sum_{|n|>C_{\rm off}N}
 |\widehat{\widetilde a_N}_n(\omega)|
 \leq C_JN^{P_a+1-J}\|f\|_{H^k}.                        \tag{4.32}
\end{equation}
Taking $J>P_a+Q+1$ proves the required offset tail for the coefficient
and hence the $L^\infty$ estimate after applying ${\cal Q}_{C_{\rm off}N}^b$.

For the angular variable, the definition of \(a_N\) in
Subsection~\ref{sec:continuous-representation} and
Fourier inversion give
\[
 \widehat a_N(\omega,\tau)=i\tau|\tau|^{d-1}
 \chi(|\tau|/N)\widehat G(\sigma_N\tau)^{-1}
 \int_{B_{R_0}}f(y)e^{-i\tau\omega\cdot y}\,dy .
\]
Thus all angular dependence occurs in the plane wave.  Write
\(r=|\tau||y|\leq2NR_0\) and, for \(y\ne0\), put \(e=\operatorname{sgn}(\tau) y/|y|\).  The
degree-\(L\) Taylor polynomial
\[
 P_L(\omega)=\sum_{\ell=0}^L
 \frac{(-ir)^\ell}{\ell!}(\omega\cdot e)^\ell
\]
belongs to \(\mathcal P_L^{\Sph}\).  If \(L+2>r\), the ratio of
successive terms in the exponential tail gives
\[
 E_L^{\Sph}(e^{-ir\omega\cdot e})
 \leq\sum_{\ell>L}\frac{r^\ell}{\ell!}
 \leq \frac{1}{1-r/(L+2)}\frac{r^{L+1}}{(L+1)!}.
\]
Take \(L=\lfloor C_0N\rfloor\) with \(C_0>2eR_0\).  The elementary
bound \((L+1)!\geq((L+1)/e)^{L+1}\) then implies, uniformly for
\(r\leq2NR_0\),
\[
 E_L^{\Sph}(e^{-ir\omega\cdot e})\leq Ce^{-cN}
 \leq C_JN^{-J}\qquad(J>0).
\]
The case \(y=0\) is constant in \(\omega\).  Applying Minkowski in
\(y,\tau\), and absorbing the polynomially growing deconvolution and
Fourier multipliers into the exponential tail, now yields uniformly in
\(b\),
\[
 E_{C_0N}^{\Sph}(a_N(\cdot,b))
 \leq C_JN^{-J}\|f\|_{H^k}
\]
for arbitrary \(J\).
This establishes the middle term of (4.28).  Combining the angular and
offset estimates yields the required approximation of $\widetilde a_N$.

It remains to treat the synthesized product.  For $|\alpha|=j\leq k$,
write
$H_{N,\alpha}=\partial_x^\alpha H_N
=c_d\omega^\alpha a_N\partial_t^jK_{\sigma_N}^\phi$.
Fourier transformation in \(b\) gives
\begin{equation}
 {\cal F}_b\{a_N(\omega,b)\partial_t^jK_{\sigma_N}^\phi(t,b)\}(\eta)
 =\frac1{2\pi}\int_{-2N}^{2N}\widehat a_N(\omega,\tau)
 \partial_t^j\widehat K_{\sigma_N}^\phi(t,\eta-\tau)\,d\tau .
 \tag{4.33}
\end{equation}
For $j=0$, use $|t\Phi_1(-it\xi)|\leq R$; for $j\geq1$, use the
corresponding derivative of (4.6).  Since
$\beta_{d,k}\leq d+1$, (4.10), (4.30), and (4.33) yield
\begin{equation}
 \sup_{\omega,\,x}\sum_{|n|>C_{\rm off}N}
 |{\cal F}_bH_{N,\alpha}(\omega,\lambda_n,x)|
 \leq CN^{-Q}\|f\|_{H^k}                               \tag{4.34}
\end{equation}
after increasing the fixed gap constant $C_{\rm off}$ if necessary.

Multiplication by $\psi_c$ preserves this order because
${\cal F}_b(\psi_cH_{N,\alpha})=(2\pi)^{-1}
\widehat\psi_c*{\cal F}_bH_{N,\alpha}$.  Fix a small $\delta_c>0$ and
split this convolution into $|\xi|\leq\delta_cN$ and
$|\xi|>\delta_cN$.  In the first part, the frequency gap used in (4.10)
is reduced only from $\pi C_{\rm off}/B-2$ to
$\pi C_{\rm off}/B-2-\delta_c$; the Gaussian and $\operatorname{sech}^2$
estimates in Lemma~\ref{lem:standard-sigmoids} therefore prove (4.34)
uniformly for these shifted samples.  In the second part, the lattice
Schwartz bound
\[
 \sup_{\eta\in\R}\sum_{\substack{n\in\mathbb Z\\
 |\lambda_n-\eta|>\delta_cN}}
 |\widehat\psi_c(\lambda_n-\eta)|\leq C_JN^{-J}
\]
is multiplied by $\|{\cal F}_bH_{N,\alpha}\|_1$.  Young's inequality,
(4.6), and (4.30) give the explicit bound
\begin{equation}
 \|{\cal F}_bH_{N,\alpha}\|_1
 \leq C\|\widehat a_N\|_1
 \begin{cases}
  1+\log(2+\sigma_N^{-1}),&j=0,\\
  \sigma_N^{-j},&1\leq j\leq k,
 \end{cases}
 \leq C N^{P_a+k+1}\|f\|_{H^k}.                         \tag{4.35}
\end{equation}
Choosing $J>P_a+k+1+Q$
proves the first term of (4.29), including the normalization factor
$(2B)^{-1}$.

In the angular variable combine the preceding approximation of \(a_N\)
with (4.11) and the product inequality
\[
 E_{L_1+L_2}^{\Sph}(UV)
 \leq\|V\|_\infty E_{L_1}^{\Sph}(U)
 +\|U\|_\infty E_{L_2}^{\Sph}(V)
 +E_{L_1}^{\Sph}(U)E_{L_2}^{\Sph}(V).
\]
The extra factor $\omega^\alpha$ has spherical degree at most $k$.
The required sup norms are bounded by (4.24) and part~1 of
Lemma~\ref{lem:standard-sigmoids}.  More precisely,
$\|a_N\|_\infty\leq CN^{P_a}\|f\|_{H^k}$, and the $N^{d+1}$ factor
in $\mathbb D_N^\phi$ dominates this loss because
$\beta_{d,k}\leq d+1$.  The polynomial local kernel bound multiplying
the superalgebraic angular tail of $a_N$ is absorbed by choosing its
order $J$ larger.  Hence, with
$C_\omega=C_0+C_{\phi,k,Q}^{\rm bw}+k$, the second term of (4.29)
follows from the displayed product inequality.

Finally set
\[
 \widetilde a_N^\sharp={\cal Q}_{C_\omega N}^\omega
 {\cal Q}_{C_{\rm off}N}^b\widetilde a_N,
 \qquad
 H_N^\sharp={\cal Q}_{C_\omega N}^\omega
 {\cal Q}_{C_{\rm off}N}^b(\psi H_N).
\]
The multiplier cutoffs are supported in twice their nominal degrees.
Thus $\widetilde a_N^\sharp\in{\cal V}_{C_*N}$ and
$H_N^\sharp\in{\cal V}_{C_*N}\otimes H^k(\Omega)$ with
$C_*=2\max\{C_{\rm off},C_\omega\}$.  Their uniform boundedness and (4.29)
give (4.28).  The $H^s(\Omega)$ norm is controlled by the finitely many
physical derivatives $|\alpha|\leq s$.  Because (4.29) already takes
the maximum over $|\alpha|\leq k$, this single $C_*$ works
simultaneously for every $s=0,\ldots,k$.
\end{proof}

The angular part of Lemma~\ref{lem:activation-bandwidth} implies the
corresponding estimate for the offset sum (4.19), since its weights have
uniformly bounded total mass.  We record the resulting statement in the form needed by
spherical designs.

\begin{lemma}[Effective angular bandwidth]
\label{lem:angular}
Let \(\phi\) satisfy Assumption~\ref{ass:activation}, with the constants
of Lemma~\ref{lem:standard-sigmoids} fixed at \(Q=Q_*\).  There is a
constant \(C_*>0\), depending only on the fixed data and independent
of \(N\), such that,
for every \(N\geq N_{\phi,k,Q_*}\),
\begin{equation}
 \max_{|\alpha|\leq k}\sup_{x\in\Omega}
 E_{C_*N}^{\Sph}
 \bigl(\partial_x^\alpha F_N^\phi(\cdot,x)\bigr)
 \leq CN^{-Q_*}\|f\|_{H^k(\R^d)}.                       \tag{4.36}
\end{equation}
\end{lemma}

\begin{proof}
First write the offset sum as
\(F_N^\phi(\omega,x)-F_N^\phi(\omega,0)\) plus its value at zero.
For the centered difference, \(\psi=1\) on \([-B_0,B_0]\), so every
node on which \(\psi\ne1\) lies in the physical tail controlled by
Lemma~\ref{lem:offset-estimates}.  The integrand estimate in (4.29) controls the remaining
nodes uniformly in \(b\).  The value at zero is a weighted sum of
\(a_N(\omega,b)\phi(-b/\sigma_N)\).  It is controlled by the uniform
coefficient angular tail in (4.28) and boundedness of \(\phi\).
Summing the best-approximation errors with weights \(h_N\) costs only
\(h_N\sum_{|\ell h_N|\leq B}1\leq C_B\), which proves (4.36).
\end{proof}

\subsection{Proof of Theorem~\ref{thm:main}}
\label{sec:deterministic-proof}

We now construct the deterministic discrete synthesis operator by combining
the uniform offset rule of Subsection~\ref{sec:offset-quadrature} with an
equal-weight spherical design.  The effective angular bandwidth from
Subsection~\ref{sec:effective-bandwidth} makes the design exact on
the relevant finite-dimensional part of the synthesis integrand.

The following standard estimate converts polynomial exactness and
\(\ell^1\)-stability of the weights into a cubature error controlled by
best uniform approximation.

\begin{lemma}
\label{lem:cubature-transfer}
Let \(\mathcal Q\) be compact with a finite Borel measure
\(\mu_{\mathcal Q}\), and define
\[
 I(H):=\int_{\mathcal Q}H\,d\mu_{\mathcal Q},
 \qquad Q_M(H):=\sum_{j=1}^M w_jH(p_j),
 \qquad \Lambda_M:=\sum_{j=1}^M|w_j|.
\]
Suppose that \(Q_M(P)=I(P)\) for every \(P\) in a finite-dimensional
space \(\mathcal W\subset C(\mathcal Q)\).  Then, for every Banach space
\(Y\) and \(H\in C(\mathcal Q;Y)\),
\begin{equation}
 \|I(H)-Q_M(H)\|_Y
 \leq\bigl(\mu_{\mathcal Q}(\mathcal Q)+\Lambda_M\bigr)
 \inf_{V\in\mathcal W\otimes Y}
 \|H-V\|_{L^\infty(\mathcal Q;Y)},                        \tag{4.37}
\end{equation}
where \(\mathcal W\otimes Y\) consists of the finite sums
\(V=\sum_{r=1}^R P_r y_r\) with \(P_r\in\mathcal W\) and \(y_r\in Y\).
\end{lemma}

\begin{proof}
For \(V\in\mathcal W\otimes Y\), scalar exactness gives \(I(V)=Q_M(V)\).
Hence
\[
\begin{aligned}
 &\|I(H)-Q_M(H)\|_Y
 =\|I(H-V)-Q_M(H-V)\|_Y\\
 &\leq \|I(H-V)\|_Y+\|Q_M(H-V)\|_Y
 \leq\bigl(\mu_{\mathcal Q}(\mathcal Q)+\Lambda_M\bigr)
 \|H-V\|_{L^\infty(\mathcal Q;Y)}.
\end{aligned}
\]
Taking the infimum over \(V\) proves (4.37).
\end{proof}

Set $t_N=\lceil C_*N\rceil$.  For $d\geq3$, Bondarenko, Radchenko, and
Viazovska~\cite{BRV} proved that, for every
$M\geq C_dt_N^{d-1}$, one can choose spherical-design nodes on
$\Sph^{d-1}$; for $d=2$, the same conclusion follows from equally
spaced points on $\Sph^1$.  Thus one can choose
$X=\{\omega_\nu\}_{\nu=1}^M\subset\Sph^{d-1}$ for which the
equal-weight rule is exact on \(\mathcal P_{t_N}^{\Sph}\):
\begin{equation}
 \frac{1}{\abs{\Sph^{d-1}}}
 \int_{\Sph^{d-1}}P(\omega)\,d\omega
 =
 \frac{1}{M}\sum_{\nu=1}^M P(\omega_\nu)
 \qquad\text{for every }P\in\mathcal P_{t_N}^{\Sph}.
 \tag{4.38}
\end{equation}
We choose such a design with
\(cN^{d-1}\leq M_{\omega,N}\leq C'_dt_N^{d-1}\leq C N^{d-1}\).
Since $t_N\geq C_*N$, (4.38) integrates every spherical polynomial of
degree at most $C_*N$ exactly.  For this positive equal-weight rule,
\(\Lambda_M=\abs{\Sph^{d-1}}\).  Lemma~\ref{lem:cubature-transfer} with
\(Y=\R\) therefore gives, for every
$U\in C(\Sph^{d-1})$,
\begin{equation}
 \left|\int_{\Sph^{d-1}}U(\omega)\,d\omega
 -\frac{\abs{\Sph^{d-1}}}{M_{\omega,N}}
 \sum_{\nu=1}^{M_{\omega,N}}U(\omega_{\nu,N})\right|
 \leq 2\abs{\Sph^{d-1}}E_{C_*N}^{\Sph}(U).                \tag{4.39}
\end{equation}
Applying (4.39) to the centered functions
$\partial_x^\alpha\overline F_N^\phi(\cdot,x)$ and then integrating in $x$
gives, for $m=0,\ldots,k$,
\begin{equation}
 \left\|c_d\int_{\Sph^{d-1}}\overline F_N^\phi(\omega,\cdot)d\omega
 -c_d\frac{\abs{\Sph^{d-1}}}{M_{\omega,N}}
 \sum_{\nu=1}^{M_{\omega,N}}
 \overline F_N^\phi(\omega_{\nu,N},\cdot)
 \right\|_{H^m(\Omega)}
 \leq C N^{-Q_*}\norm f_{H^k}.
 \tag{4.40}
\end{equation}

Set \(L_N=\lfloor B/h_N\rfloor\),
\(a_{\nu\ell,N}:=
\frac{c_d\abs{\Sph^{d-1}}}{M_{\omega,N}}
h_Na_N(\omega_{\nu,N},b_{\ell,N})\), and
define the fully discrete output constant by
\begin{equation}
 a_{0,N}:=f_N(0)-
 c_d\frac{\abs{\Sph^{d-1}}}{M_{\omega,N}}
 \sum_{\nu=1}^{M_{\omega,N}}F_N^\phi(\omega_{\nu,N},0).
 \tag{4.41}
\end{equation}
These coefficients depend linearly on the fixed extension $f=Eu$, and
substitution in (4.19) shows that the network
\eqref{eq:det-network} is exactly the
centered discrete synthesis operator introduced in
Section~\ref{sec:main-results}.
Combining (4.2), (4.27), and (4.40), and then restricting from $\R^d$ to
$\Omega$, gives, for
$m=0,\ldots,k$,
\[
 \norm{u-v_N}_{H^m(\Omega)}
 \leq C\bigl(N^{m-k}+N^{-Q_*}\bigr)\norm u_{H^k(\Omega)}
 \leq C N^{m-k}\norm u_{H^k(\Omega)}.
\]
Thus \eqref{eq:det-rate} holds simultaneously in \(m\), with \(m=0\)
giving the stated \(L^2\) rate.
The $O(N^{d-1})$ directions and $O(N)$ offsets in $[-B,B]$ give
$M_N=O(N^d)$.

Finally, Fourier inversion and (4.24) give
\[
 \abs{a_N(\omega,b)}
 \leq C N^{P_{\rm inv}+\beta_{d,k}}\norm f_{H^k}.
\]
Since $M_{\omega,N}\asymp N^{d-1}$ and $h_N\asymp N^{-1}$,
the definition of $a_{\nu\ell,N}$ yields
\[
 \abs{a_{\nu\ell,N}}
 \leq C N^{P_{\rm inv}+\beta_{d,k}-d}\norm u_{H^k(\Omega)}.
\]
This implies the non-sharp bound \eqref{eq:det-coeff-bound}, for example with
\(\Gamma_{\phi,d,k}=1+P_{\rm inv}+\beta_{d,k}\).  The preceding linearity
completes the proof of Theorem~\ref{thm:main} for the fixed activation
and scale.

\subsection{Random parameter-space cubature}
\label{sec:parameter-bandwidth}

For random sampling we work on
\(\Theta=\Sph^{d-1}\times\Sph_B^1\), with the product measure
\(\mu\), and use the space
\({\cal V}_L=\mathcal P_L^{\Sph}\otimes{\cal T}_L^B\) introduced in
Subsection~\ref{sec:effective-bandwidth}.  Recall
\(\widetilde a_N=\psi a_N\) and
\[
 H_N((\omega,b),x)=
 c_d a_N(\omega,b)K_{\sigma_N}^\phi(\omega\cdot x,b).
\]
Apply Lemma~\ref{lem:activation-bandwidth} with the fixed activation and
scale at order \(Q_*=k+d+1\).  Its \(N^{-Q_*}\) estimate gives both error orders
used below.  Thus there exist a fixed constant \(C_*>0\), an integrand model
\(H_N^\sharp\in{\cal V}_{C_*N}\otimes H^k(\Omega)\), and a coefficient model
\(\widetilde a_N^\sharp\in{\cal V}_{C_*N}\) such that, simultaneously
for \(m=0,\ldots,k\),
\begin{equation}
 \sup_{p\in\Theta}
 \|\psi H_N(p,\cdot)-H_N^\sharp(p,\cdot)\|_{H^m(\Omega)}
 \leq C N^{-k}\|f\|_{H^k(\R^d)},                         \tag{4.42}
\end{equation}
and
\begin{equation}
 \|\widetilde a_N-\widetilde a_N^\sharp\|_{L^\infty(\Theta)}
 \leq C N^{-k-d-1}\|f\|_{H^k(\R^d)}.                   \tag{4.43}
\end{equation}
The coefficient estimate from Lemma~\ref{lem:offset-estimates} is
\begin{equation}
 \|\widetilde a_N\|_{L^2(\Theta)}
 \leq C N^{\gamma_{\phi,d,k}}\|f\|_{H^k(\R^d)}.\tag{4.44}
\end{equation}
We now turn this finite-dimensional parameter model into a random discrete
synthesis operator.  We use Gram-matrix
concentration to establish stable sampling of ${\cal V}_L$ from the
prescribed density and construct bounded exact integration weights.

Retain the surface-times-Lebesgue measure \(\mu\) fixed in
Section~\ref{sec:main-results}, put
\(Z:=\mu(\Theta)=2B|\Sph^{d-1}|\), and choose a real orthonormal basis
$\{e_r\}_{r=1}^{D_L}$ of ${\cal V}_L$ in $L^2(\mu)$.  Write
$\mathsf e(p)=(e_1(p),\ldots,e_{D_L}(p))^{\mathsf T}$.  Because
${\cal V}_L$ is the product of the complete spherical-polynomial space
and the complete trigonometric-polynomial space, rotation and translation
invariance imply that its Christoffel function is constant:
\begin{equation}
 {\cal K}_L(p):=\|\mathsf e(p)\|_{\ell^2}^2
 =\sum_{r=1}^{D_L}|e_r(p)|^2=\frac{D_L}{Z}.
 \tag{4.45}
\end{equation}
Indeed, the transitive action of
$SO(d)\times\Sph_B^1$ preserves both $\mu$ and ${\cal V}_L$, so the
diagonal of the projection kernel is constant.  Integrating it over
\(\Theta\) gives
$Z{\cal K}_L=\sum_{r=1}^{D_L}\|e_r\|_{L^2(\mu)}^2=D_L$.
Let $p_1,\ldots,p_M$ be independent with law
$d\nu=\rho\,d\mu$, where $\rho\geq\rho_0>0$ $\mu$-a.e., and form the
$M\times D_L$ weighted sampling matrix
\({\bf A}_{jr}:=e_r(p_j)/\sqrt{M\rho(p_j)}\).  Its Gram matrix
satisfies
\(\mathbb E({\bf A}^*{\bf A})
=\int_\Theta\overline{\mathsf e(p)}\mathsf e(p)^{\mathsf T}\,d\mu(p)
=I_{D_L}\).

\begin{lemma}[Random Gram-matrix concentration~\cite{Tropp}]
\label{lem:random-gram}
For every $0<\varepsilon<1$,
\begin{equation}
 \mathbb P\{\|{\bf A}^*{\bf A}-I_{D_L}\|>\varepsilon\}
 \leq 2D_L e^{-c\varepsilon^2M/D_L},
 \tag{4.46}
\end{equation}
where $c>0$ depends only on $Z\rho_0$.  On the complementary event,
\begin{equation}
 (1-\varepsilon)\|P\|_{L^2(\mu)}^2
 \leq\frac1M\sum_{j=1}^M
       \frac{|P(p_j)|^2}{\rho(p_j)}
 \leq(1+\varepsilon)\|P\|_{L^2(\mu)}^2
 \quad(P\in{\cal V}_L).
 \tag{4.47}
\end{equation}
\end{lemma}

\begin{proof}
Put \(\mathbf Z_j:=\rho(p_j)^{-1}
\overline{\mathsf e(p_j)}\mathsf e(p_j)^{\mathsf T}\).
These are independent positive semidefinite matrices,
$\mathbb E \mathbf Z_j=I_{D_L}$, and, by (4.45),
\(0\leq \mathbf Z_j\leq (Z\rho_0)^{-1}D_L I_{D_L}\).
The matrix Chernoff inequality applied to
$M^{-1}\sum_j\mathbf Z_j={\bf A}^*{\bf A}$ gives
\(\mathbb P\{\lambda_{\min}({\bf A}^*{\bf A})<1-\varepsilon\}
\leq D_L e^{-c\varepsilon^2M/D_L}\),
and the same bound, with a possibly different constant $c>0$, for
$\lambda_{\max}>1+\varepsilon$.  Their union is (4.46).
If $P=\sum_r\xi_re_r$ and
$\boldsymbol\xi=(\xi_1,\ldots,\xi_{D_L})^{\mathsf T}$, then
$\|P\|_2^2=\|\boldsymbol\xi\|_2^2$ and the empirical expression in
(4.47) is $\|{\bf A}\boldsymbol\xi\|_2^2$.  Thus the two spectral bounds
are exactly (4.47).
\end{proof}

\begin{lemma}
\label{lem:random-weights}
On the event (4.47), there are weights $w_1,\ldots,w_M$ such that
\begin{equation}
 \int_\Theta P\,d\mu=\sum_{j=1}^Mw_jP(p_j)
 \quad(P\in{\cal V}_L),
 \qquad
 \sum_{j=1}^M|w_j|
 \leq\left(\frac{Z}{\rho_0(1-\varepsilon)}\right)^{1/2}.
 \tag{4.48}
\end{equation}
\end{lemma}

\begin{proof}
Regard ${\bf A}:{\cal V}_L\to\mathbb C^M$ as
$({\bf A}P)_j=P(p_j)/\sqrt{M\rho(p_j)}$.
By (4.47), ${\bf A}^*{\bf A}$ is invertible.  Since the Riesz
representer of $P\mapsto\int P\,d\mu$ is the constant function
${\bf1}$, define
\begin{equation}
 z={\bf A}({\bf A}^*{\bf A})^{-1}{\bf1}.
 \tag{4.49}
\end{equation}
Then ${\bf A}^*z={\bf1}$ and, under the convention that the inner product
is linear in its first argument,
\[
 \int P\,d\mu=\langle P,{\bf1}\rangle
 =\langle{\bf A}P,z\rangle
 =\sum_{j=1}^M
   \frac{\overline{z_j}}{\sqrt{M\rho(p_j)}}P(p_j).
\]
Thus $w_j=\overline{z_j}/\sqrt{M\rho(p_j)}$.  Moreover,
\begin{equation}
 \|z\|_2^2
 =\langle{\bf1},({\bf A}^*{\bf A})^{-1}{\bf1}\rangle
 \leq(1-\varepsilon)^{-1}\|{\bf1}\|_2^2
 =Z(1-\varepsilon)^{-1}.
 \tag{4.50}
\end{equation}
Cauchy--Schwarz and $\rho(p_j)\geq\rho_0$ now give
\(\sum_j|w_j|\leq\|z\|_2
(\frac1M\sum_j\rho(p_j)^{-1})^{1/2}
\leq[Z/(\rho_0(1-\varepsilon))]^{1/2}\), which proves (4.48).
Because the basis and the sampling factors are real, the matrices in
(4.49), the vector $z$, and hence the weights $w_j$ are real.
\end{proof}

\begin{proof}[Proof of Theorem~\ref{thm:random}]
Let \(C_*\) be the bandwidth constant in (4.42)--(4.43) and put
\(L=C_*N\).  Since $D_L\leq C L^d\leq C N^d$,
Lemma~\ref{lem:random-gram} with
$\varepsilon=1/2$ gives an event whose failure probability is at most
$CN^d e^{-cM/N^d}$.  The assumption
$M\geq C N^d\log(N/\delta)$ makes this at most $\delta$.
The event depends only on the sample and is therefore simultaneous in
$u$.

On this event use the weights of Lemma~\ref{lem:random-weights}.
They integrate $H_N^\sharp(\cdot,x)\in{\cal V}_L$ exactly for every
$x$.  First replace the integral over \({\cal C}\) by the integral of
\(\psi H_N\) over \(\Theta\).  Since \(1-\psi\) is supported where
\(|b|>B_0\), the physical-tail part of
Lemma~\ref{lem:offset-estimates} makes the discarded tail
superalgebraically small.  Apply
Lemma~\ref{lem:cubature-transfer} with
\(\mathcal Q=\Theta\), \(\mathcal W={\cal V}_L\),
\(Y=H^m(\Omega)\), and approximant \(H_N^\sharp\).  Equations (4.42)
and (4.48) then give,
simultaneously for \(m=0,\ldots,k\),
\begin{equation}
 \left\|\int_{\cal C}H_N(p,\cdot)\,d\mu(p)
 -\sum_{j=1}^Mw_j\psi(b_j)H_N(p_j,\cdot)\right\|_{H^m(\Omega)}
 \leq CN^{-k}\|u\|_{H^k(\Omega)}.
 \tag{4.51}
\end{equation}
By (4.4), the fact that \(1-\psi\) is supported in \(|b|>B_0\), and the
physical-tail conclusion of Lemma~\ref{lem:offset-estimates}, the integral in (4.51) equals
$f_N-f_N(0)$ up to an
$O(N^{-k})\|u\|_{H^k}$ remainder in every $H^m(\Omega)$, $m\leq k$.
Lemma~\ref{lem:cutoff} contributes
$CN^{m-k}\|u\|_{H^k}$ in $H^m$.  Since $N^{-k}\leq N^{m-k}$, this proves
\eqref{eq:random-N-rate}, including its \(m=0\) case.  Finally, set
\(a_j:=c_dw_j\widetilde a_N(p_j)\).
Using the representation
$w_j=\overline{z_j}/\sqrt{M\rho(p_j)}$ in
Lemma~\ref{lem:random-weights}, Cauchy--Schwarz gives
\[
 \sum_{j=1}^M|a_j|
 \leq C\|z\|_2
 \left(\frac1M\sum_{j=1}^M
 \frac{|\widetilde a_N(p_j)|^2}{\rho(p_j)}\right)^{1/2}.
\]
Insert $\widetilde a_N^\sharp$ from (4.43).  The Gram estimate (4.47)
controls its empirical norm, while (4.43) and $\rho\geq\rho_0$ control the
remainder.  More explicitly, with
$r_N=\widetilde a_N-\widetilde a_N^\sharp$, Minkowski's inequality,
(4.47), and $\rho\geq\rho_0$ give
\[
 \begin{aligned}
 \left(\frac1M\sum_{j=1}^M
 \frac{|\widetilde a_N(p_j)|^2}{\rho(p_j)}\right)^{1/2}
 &\leq (1+\varepsilon)^{1/2}
 \|\widetilde a_N^\sharp\|_{L^2(\mu)}
 +\rho_0^{-1/2}\|r_N\|_{L^\infty(\Theta)}\\
 &\leq C\|\widetilde a_N\|_{L^2(\mu)}
 +C\|r_N\|_{L^\infty(\Theta)}.
 \end{aligned}
\]
The second inequality also uses
$\|r_N\|_{L^2(\mu)}\leq Z^{1/2}\|r_N\|_{L^\infty(\Theta)}$.
Equations (4.43), (4.44), and (4.50) therefore imply
\(\sum_{j=1}^M|a_j|
\leq C N^{\gamma_{\phi,d,k}}\|u\|_{H^k(\Omega)}\).
The corresponding centered constant is
\(a_0:=f_N(0)-\sum_{j=1}^M a_j\phi(-b_j/\sigma_N)\).
Fourier inversion, Cauchy--Schwarz, and
$\operatorname{supp}\widehat f_N\subset B_{2N}$ give
\[
 |f_N(0)|\leq C\left(\int_{|\zeta|\leq2N}(1+|\zeta|^2)^{-k}\,d\zeta\right)^{1/2}
 \|f\|_{H^k(\R^d)}\leq C N^{\gamma_{\phi,d,k}}\|u\|_{H^k(\Omega)}.
\]
Indeed, the integral is bounded, grows like $\log N$, or grows like
$N^{d-2k}$ according as $k>d/2$, $k=d/2$, or $k<d/2$; the definition
of $\gamma_{\phi,d,k}$ and \(P_{\rm inv}\geq0\) dominate
all three cases.
Since \(\phi\) is bounded, it follows that
\(|a_0|\leq |f_N(0)|+\sum_{j=1}^M|a_j|
\leq C N^{\gamma_{\phi,d,k}}\|u\|_{H^k(\Omega)}\).
Thus \eqref{eq:random-coeff-bound} follows.  Provided the sampling budget satisfies
\(M\geq C N_0^d\log(N_0/\delta)\), the largest integer \(N\geq N_0\)
satisfying \(M\geq C N^d\log(N/\delta)\) obeys
\(N\geq c(M/\log(M/\delta))^{1/d}\),
because, for a sufficiently small fixed $c>0$, the integer part of the
right-hand side satisfies the same inequality.  Here
$\log(N/\delta)\leq C\log(M/\delta)$ for $N\leq M$ and $0<\delta<1$.
The weights depend only on the sample, whereas $f_N$ and $a_N$ depend
linearly on $u$, so the resulting approximation operator is linear.
This gives \eqref{eq:random-width-rate} and completes the proof of
Theorem~\ref{thm:random} for the fixed activation and scale.
\end{proof}

\appendix
\section{Numerical experiment settings}
\label{app:numerical-settings}

\subsection{Targets and error measurement}

All experiments use \(\Omega=[0,1]^d\), with centered coordinates for the
ten-dimensional feature evaluation.  The targets are finite random Fourier
series with controlled Sobolev regularity.  For a
finite set \(\Lambda\subset\mathbb Z^d\setminus\{0\}\) containing one
representative of each retained frequency pair, we set
\begin{equation}
 \widetilde u_{k,\Lambda}(x)=
 \sum_{n\in\Lambda}
 \frac{a_n\cos(2\pi n\cdot x)+b_n\sin(2\pi n\cdot x)}
 {(1+|n|^2)^{(k+d/2)/2}\log(2+|n|)},
 \qquad a_n,b_n\stackrel{\rm iid}{\sim}{\cal N}(0,1).
 \tag{A.1}
\end{equation}
The coefficient decay corresponds to the \(H^k\) regularity threshold.  All
targets are normalized before fitting, and each normalized target is denoted by
\(u\).

For a fitted approximation \(v_M\), the reported relative \(H^m\) error is
\begin{equation}
 \mathcal E_m(v_M;u)
 :=\frac{\|u-v_M\|_{H^m(\Omega)}}{\|u\|_{H^m(\Omega)}},
 \qquad m=0,1,2,
 \tag{A.2}
\end{equation}
Thus \(m=0\) gives the relative \(L^2\) error.  Each target and frequency set
is held fixed across widths.  The discrete \(L^2\) least-squares systems are
assembled from physical-space point sets chosen independently of the targets
and feature dictionaries.  In two dimensions, the deterministic experiment in
Figure~\ref{fig:num-2d} uses a \(129\times129\) midpoint grid; for the random
experiments, we use \(n_{\mathrm{tr}}^2\) tensor-product points, where
\(n_{\mathrm{tr}}\) is the smallest odd integer not less than
\(\max\{65,\sqrt{2(M+1)}\}\).  Figure~\ref{fig:num-2d} uses midpoint grids,
whereas Figure~\ref{fig:num-sobolev-2d} uses endpoint grids with tensor-product
Simpson weights.  The three-dimensional deterministic experiment likewise uses
tensor-product Simpson quadrature, while the ten-dimensional experiments use
fixed scrambled Sobol point sets; in all cases, the number of training points
exceeds the dictionary width.  For the \(H^1\) and \(H^2\) tests, the same
\(L^2\)-fitted coefficients are retained, and all derivatives used to evaluate
the Sobolev errors are computed analytically.

Empirical orders are the negative least-squares slopes of
\(\log\mathcal E_m\) against \(\log W\), with \(W=M\) for deterministic
dictionaries and \(W=M/\log(N/\delta)\) for random dictionaries.  All
displayed widths enter the fit.  Figure~\ref{fig:num-2d} reports medians and
interquartile ranges over ten target (deterministic) or dictionary (random)
realizations.  Because the ten-dimensional least-squares solves are
substantially more expensive, the right panel of
Figure~\ref{fig:num-3d-10d} uses three independent target seeds, shared across
all four regularities, and reports the corresponding medians and interquartile
ranges.

\subsection{Feature dictionaries}

All offsets lie in \([-2,2]\).  The deterministic constructions in
Figure~\ref{fig:num-2d} and the left panel of
Figure~\ref{fig:num-3d-10d} use target-independent, uniform direction--offset
grids with \(M=N^d\).  The random constructions use
\[
 M=\lceil N^d\log(N/\delta)\rceil,\qquad \delta=0.01,
\]
and their errors are plotted against the effective width
\(M/\log(N/\delta)\).

All random direction--offset pairs are sampled independently from fixed
target-independent distributions whose densities are bounded away from zero.
We use the unit-prefactor scales \(\sigma_N=\log(N)/N\) for \(\tanh\) and
\(\sigma_N=\sqrt{\log N}/N\) for \(\erf\). For the three-dimensional deterministic experiment, we use
\(\sigma_N=\sqrt{4\log N}/N\) to mitigate pre-asymptotic effects while
preserving the theoretical scaling.

\begin{table}[H]
\setlength{\abovecaptionskip}{4pt}
\setlength{\belowcaptionskip}{4pt}
\caption{Empirical convergence orders by problem setting
(predicted values in parentheses).}
\label{tab:empirical-orders}
\centering
\scriptsize
\setlength{\tabcolsep}{5pt}
\renewcommand{\arraystretch}{0.90}
\begin{tabular}{@{}ccccc@{\hspace{8pt}}c@{}}
\toprule
$d$ & Error & Activation & Dictionary & $k$ & Empirical order \\
\midrule
\multirow{4}{*}{$2$}
 & \multirow{2}{*}{$L^2$}
 & \multirow{2}{*}{$\tanh$}
 & deterministic & $2,4,6$ & $1.17\,(1),\ 2.16\,(2),\ 3.19\,(3)$ \\
 & & & random & $2,4,6$ & $0.86\,(1),\ 1.79\,(2),\ 2.82\,(3)$ \\
\cmidrule(l){2-6}
 & $H^1$
 & \multirow{2}{*}{$\erf$}
 & \multirow{2}{*}{random}
 & $2,4,6$ & $0.66\,(0.5),\ 1.59\,(1.5),\ 2.66\,(2.5)$ \\
 & $H^2$ & & & $3,4,6$ & $0.65\,(0.5),\ 1.03\,(1),\ 2.12\,(2)$ \\
\midrule
$3$  & $L^2$ & $\erf$  & deterministic & $2,4,6$    & $0.75\,(0.67),\ 1.44\,(1.33),\ 2.02\,(2)$ \\
$10$ & $L^2$ & $\erf$  & random        & $2,5,10,20$ & $0.37\,(0.2),\ 0.68\,(0.5),\ 1.00\,(1),\ 1.96\,(2)$ \\
\bottomrule
\end{tabular}
\end{table}

\end{document}